\documentclass[12pt]{amsart}

\usepackage{calligra,mathrsfs}
\usepackage[all]{xy}
\usepackage{float, comment}
\usepackage{mathtools}
\usepackage{amsmath}
\usepackage{amsthm}
\usepackage{amssymb}
\usepackage{amsbsy}
\usepackage{amstext}
\usepackage{amsopn}
\usepackage{mathrsfs} 
\usepackage{enumerate}
\usepackage{xcolor}
\usepackage{graphicx} 
\usepackage{microtype} 
\usepackage[margin=1in,marginparwidth=0.8in, marginparsep=0.1in]{geometry}
\usepackage[pagebackref, bookmarks=true, bookmarksopen=true, bookmarksdepth=3,bookmarksopenlevel=2, colorlinks=true, linkcolor=blue, citecolor=blue, filecolor=blue, menucolor=blue, urlcolor=blue]{hyperref}
\usepackage{tikz}
\usepackage{bbm}
\usepackage[all]{xy}
\usetikzlibrary{arrows,calc,positioning,decorations.pathreplacing} 

\usepackage{stackrel}

\usepackage[shortlabels]{enumitem}

\numberwithin{equation}{section}

\usepackage{todonotes}

\usepackage{xcolor}

\usepackage{stmaryrd}

\usepackage{tikz-cd}

\usepackage{enumitem}

\numberwithin{equation}{section}
\newtheorem{thm}{Theorem}[section]
\newtheorem{prop}[thm]{Proposition}
\newtheorem{lem}[thm]{Lemma}
\newtheorem{cor}[thm]{Corollary}
\newtheorem{rem}[thm]{Remark}
\newtheorem{definition}[thm]{Definition}
\newtheorem{example}[thm]{Example}

\newtheorem*{cor*}{Corollary}

\newcommand{\nc}{\newcommand}

\nc{\bA}{\mathbb A}
\nc{\bC}{\mathbb C}
\nc{\bc}{{\bf c}}
\nc{\bD}{\mathbb D}
\nc{\bd}{\mathbb d}
\nc{\bG}{\mathbb G}
\nc{\bi}{\bold i}
\nc{\bL}{\mathbb L}
\nc{\bN}{\mathbb N}
\nc{\bP}{\mathbb P}
\nc{\bQ}{\mathbb Q}
\nc{\bR}{\mathbb R}
\nc{\bu}{\mathbb u}
\nc{\bv}{\bold v}
\nc{\bw}{\bold w}
\nc{\bW}{\mathbb W}
\nc{\bZ}{\mathbb Z}

\nc{\cA}{\mathcal A}
\nc{\cC}{\mathcal C}
\nc{\cD}{\mathcal D}
\nc{\cE}{\mathcal E}
\nc{\cF}{\mathcal F}
\nc{\cG}{\mathcal G}
\nc{\cH}{\mathcal H}
\nc{\cI}{\mathcal I}
\nc{\cK}{\mathcal K}
\nc{\cL}{\mathcal L}
\nc{\cN}{\mathcal N}
\nc{\cO}{\mathcal O}
\nc{\cP}{\mathcal P}
\nc{\cT}{\mathcal T}
\nc{\cU}{\mathcal U}
\nc{\cV}{\mathcal V}
\nc{\cW}{\mathcal W}

\nc{\al}{\alpha}
\nc{\be}{\beta}
\nc{\la}{\lambda}
\nc{\La}{\Lambda}
\nc{\ve}{\varepsilon}
\nc{\om}{\omega}
\nc{\Om}{\Omega}

\nc{\gl}{\mathfrak{gl}}
\nc{\fsl}{\mathfrak{sl}}
\nc{\ff}{\mathfrak{f}}
\nc{\g}{\mathfrak{g}}
\nc{\gh}{\widehat\g}
\nc{\h}{\mathfrak{h}}
\nc{\fl}{\mathfrak{l}}
\nc{\fm}{{\mathfrak m}}
\nc{\fM}{{\mathfrak M}}
\nc{\fp}{{\mathfrak p}}
\nc{\fh}{{\mathfrak h}}
\nc{\fg}{{\mathfrak g}}
\nc{\fgh}{{\widehat{\mathfrak g}}}
\nc{\fb}{{\mathfrak b}}
\nc{\fn}{{\mathfrak n}}
\nc{\fQ}{\mathfrak{Q}}
\nc{\ft}{\mathfrak t}

\nc{\Aut}{\operatorname{Aut}}
\nc{\ch}{{\mathop {\rm ch}}}
\nc{\tr}{{\mathop {\rm tr}\,}}
\nc{\im}{{\mathop {\rm im}}}
\nc{\id}{{\mathop {\rm id}}}
\nc{\ad}{{\mathop {\rm ad}}}
\nc{\gr}{\mathrm{gr}}
\nc{\ord}{\mathrm{ord}}
\nc{\red}{\mathrm{red}}
\nc{\End}{\operatorname{End}}
\nc{\Spec}{\operatorname{Spec}}
\nc{\Spf}{\operatorname{Spf}}
\nc{\Proj}{\operatorname{Proj}}
\nc{\Pic}{\operatorname{Pic}}
\nc{\Lie}{\operatorname{Lie}}
\nc{\Der}{\operatorname{Der}}
\nc{\Coh}{\mathrm{Coh}}
\nc{\coh}{\mathrm{coh}}
\nc{\qcoh}{\mathrm{Qcoh }}
\nc{\Gal}{\operatorname{Gal}}
\nc{\Hom}{\mathrm{Hom}}
\nc{\Rhom}{\mathrm{RHom}}
\nc{\cHom}{\mathop{\mathcal{H}\! \mathit{om}}\nolimits}
\nc{\Ext}{\mathrm{Ext}}
\nc{\Ann}{\mathrm{Ann}}
\nc{\Vect}{\mathrm{Vect}}
\nc{\wt}{\mathrm{wt}}
\nc{\hw}{\mathrm{hw}}
\nc{\rk}{\operatorname{rank}}
\nc{\Gr}{{\mathrm {Gr}}}
\nc{\Fl}{\mathrm{Fl}}
\nc{\spn}{\mathrm{span}}
\nc{\Rep}{\operatorname{Rep}}
\nc{\Irrep}{\mathrm{Irrep }}
\nc{\supp}{\operatorname{supp}}
\nc{\tp}{\mathrm{top}}
\nc{\codim}{\mathrm{codim}}
\nc{\IC}{\operatorname{IC}}
\nc{\Res}{\mathrm{Res}}
\nc{\modules}{\mathrm{-mod}}
\nc{\Perv}{\mathrm{Perv}}
\nc{\Forg}{\operatorname{Forg}}
\nc{\shhom}{\mathop{\mathcal{H}\! \mathit{om}}\nolimits}
\nc{\chom}{\mathop{\mathcal{H}\! \mathit{om}}\nolimits}
\nc{\cEnd}{\mathop{\mathcal{E}\! \mathit{nd}}\nolimits}
\nc{\mods}{\mathrm{-mod}}
\nc{\exo}{\mathrm{exo}}
\nc{\Hol}{\operatorname{Hol}}
\nc{\torus}{\bC^\times}
\nc{\an}{\mathrm{an}}
\nc{\Out}{\operatorname{Out}}
\nc{\Tan}{\operatorname{Tan}}
\nc{\pt}{\mathrm{pt}}

\nc{\GfL}{{(G \times \bC^\times, \fl \oplus \bC D)}}
\nc{\Gfl}{{(G \times \bC^\times, \fl \oplus \bC D)}}

\nc{\eO}{\EuScript{O}}

\nc{\bra}{\langle}
\nc{\ket}{\rangle}
\nc{\pa}{\partial}
\nc{\ld}{\ldots}
\nc{\cd}{\cdots}
\nc{\hk}{\hookrightarrow}
\nc{\T}{\otimes}
\nc{\ov}{\overline}
\nc{\wh}{\widehat}
\nc{\wti}{\widetilde}
\nc{\svee}{{\!\scriptscriptstyle\vee}}
\nc{\ula}{{\underline{\la}}}
\nc{\umu}{{\underline{\mu}}}
\nc{\conv}{{\widetilde \times}}
\nc{\lach}{{\la^\svee}}
\nc{\alch}{{\al^\svee}}
\nc{\omch}{{\omega^\svee}}
\nc{\much}{{\mu^\svee}}
\nc{\md}{\text {-mod}}
\nc{\et}{\mathrm{\acute{e}t}}

\nc{\GL}{\mathfrak{GL}}
\nc{\Tr}{{\mathop {\rm Tr}\,}}
\nc{\Id}{{\mathop {\rm Id}}}

\nc{\msl}{\mathfrak{sl}}
\nc{\mgl}{\mathfrak{gl}}

\nc{\U}{\mathrm U}

\nc{\Q}{\mathfrak Q}

\nc{\on}{\operatorname} \nc\ol{\overline} \nc\ul{\underline}

\nc{\BA}{{\mathbb{A}}} \nc{\BC}{{\mathbb{C}}} \nc{\BF}{{\mathbb{F}}}
\nc{\BD}{{\mathbb{D}}} \nc{\BG}{{\mathbb{G}}} \nc{\BQ}{{\mathbb{Q}}}
\nc{\BM}{{\mathbb{M}}} \nc{\BN}{{\mathbb{N}}} \nc{\BO}{{\mathbb{O}}}
\nc{\BP}{{\mathbb{P}}} \nc{\BR}{{\mathbb{R}}}
\nc{\BZ}{{\mathbb{Z}}} \nc{\BS}{{\mathbb{S}}} \nc{\BW}{{\mathbb{W}}}

\nc{\CA}{{\mathcal{A}}} \nc{\CL}{{\mathcal{L}}} \nc{\CV}{{\mathcal{V}}} \nc{\CW}{{\mathcal{W}}}
\nc{\CalD}{{\mathcal{D}}}

\nc{\sic}{{\on{sc}}}
\nc{\add}{{\on{add}}}

\title{Perverse coherent sheaves on symplectic singularities}

\author{Ilya Dumanski}
\address{Ilya Dumanski:\newline
		Department of Mathematics, MIT, Cambridge, MA 02139, USA,
		{\it and }\newline
		Department of Mathematics, National Research University Higher School of Economics, Russian Federation,
		Usacheva str. 6, 119048, Moscow.
	}
	\email{ilyadumnsk@gmail.com}

\begin{document}





\begin{abstract}
We propose the notion of perverse coherent sheaves for symplectic singularities and study its properties. 
In particular, it gives a basis of simple objects in the Grothendieck group of Poisson sheaves. We show that perverse coherent bases for the nilpotent cone and for the affine Grassmannian arise as particular cases of our construction.
\end{abstract}

\maketitle

\tableofcontents

\section{Introduction}
\subsection{Perverse coherent sheaves} \label{intro subsec: perverse coherent start}
Perverse sheaves were introduced by Beilinson--Bernstein--Deligne in \cite{BBD82} and are ubiquitous in geometric representation theory.
Bezrukavnikov and Arinkin proposed the coherent counterpart of this notion in \cite{Bez00, AB10}. 
	
Perverse coherent sheaves is a much more restrictive notion than their original constructible version. The reason is, perverse coherent sheaves behave nicely (the IC-extension functor is defined) only in the case when one considers the category of coherent sheaves equivariant under an algebraic group action, and, moreover, this action has a finite number of orbits such that the dimensions of adjacent orbits differ at least by 2, see \cite{Bez00} (or, in the language of \cite{AB10}, we deal with a stack with a finite number of geometric points, and the perversity function on the space of points is strictly monotone and strictly comonotone).

There are two main examples of such situation in geometric representation theory: affine Grassmannian $\Gr_G$ with the action of the current group $G(\cO)$, and the nilpotent cone $\cN$ with action of the group $G$. The category of perverse coherent sheaves is a fruitful object of study in both of these instances. We now briefly sketch recent developments in these areas.

The category of perverse coherent $G(\cO)$- (or $G(\cO) \rtimes \torus$)-equivariant sheaves on $\Gr_G$ is called the {\it coherent Satake category}. Its study was initiated in \cite{BFM05}. In~\cite{CW19}, this category was connected to line defects in $4d$ $\cN = 2$ pure gauge theory (this later was extended to arbitrary gauge theory of cotangent type in~\cite{CW23}). It was also proved for $G = GL_n$ in~\cite{CW19} that this category is a cluster monoidal categorification. In \cite{Dum24} we suggested partial progress towards the same result for arbitrary simply-laced $G$. In \cite{FF21}, the basis of simple perverse coherent sheaves in $K^{GL_n(\cO) \rtimes \torus}(\Gr_{GL_n})$ was related to Lusztig’s dual canonical basis (see Section~\ref{subsec: affine grassmannian slices} for discussion on conjectural generalization of this result to other types).


The category of perverse coherent $G$- (or $G \times \torus$)-equivariant sheaves on the nilpotent cone $\cN \subset \g$ of a semisimple Lie algebra has a long history of study. 
It was used in \cite{Bez03} to establish the Lusztig--Vogan bijection, originally conjectured in \cite{Lus85}, \cite{Vog98}.
It was used in \cite{Bez06a} to describe the cohomology of the small quantum group at the root of~1, with coefficients in a tilting module, and prove a conjecture of \cite{Hum95}. Both these applications can be conceptually explained by the following facts, which are actually the main reasons for interest in the category under discussion.
The perverse t-structure on $\cN$ is closely related (and may be used to define) the so-called exotic t-structure on the category of equivariant sheaves on $\wti \cN \simeq T^* G/B$, as well as on $\wti \cN \times_\cN \wti \cN$. 
As proved in \cite{BM13}, the exotic t-structure is closely related to localization of $\g$-modules in positive characteristic \cite{BMRR08}, and it is responsible for the existence of canonical bases in K-theory, conjectured by Lusztig in \cite{Lus99}. 
Moreover, the exotic t-structure plays an important role in Bezrukavnikov's theory of coherent--constructible equivalences. 
Specifically, this is the t-structure, corresponding to the (constructible) perverse t-structure on the affine flag variety of the Langlands-dual group $G^\vee$ under Bezrukavnikov's equivalences (see \cite[Theorem 2]{Bez09}, \cite[6.2.1]{BM13}, \cite[Theorem 54]{Bez16}).
It follows that the basis of simple perverse coherent sheaves in $K^{G \times \torus}(\cN)$ is a part of the Kazhdan--Lusztig canonical basis for the affine Hecke algebra of $G^\vee$. See \cite{Ach12, ACR18, AHR22} for other results in the area.

\subsection{Symplectic singularities}
One of the leading slogans of geometric representation theory of the past decade is ``to generalize known results from the case of the nilpotent cone to other conical symplectic singularities'', (see e.g. \cite{BPW12, BLPW14}).
This class of varieties includes Kleinian singularities, hypertoric varieties, Hilbert schemes of points, slices in affine Grassmannian, Nakajima quiver varieties, as well as Higgs and Coloumb branches of 3-dimensional supersymmetric gauge theories, which provide physical background and motivation for the area (see \cite{WY23} for an overview of physical background and connection to 3d-mirror symmetry).

Thus, it is natural to seek a suitable perverse coherent t-structure for a general symplectic singularity. This endeavor can be viewed as a step towards building a theory of canonical bases for general symplectic singularities, as well as a tiny step towards understanding the extent to which Bezrukavnikov's theory has a place for symplectic singularities beyond the nilcone. 

There is, however, an obvious obstruction in building this theory for an arbitrary symplectic singularity. 
As we explained above, the existing definition works only for the case of group-equivariant sheaves, s.t. the group has a finite number of orbits with dimensions of adjacent orbits differing at least by 2. 
Such an action does not exist for symplectic singularities other than closures of nilpotent orbits and their coverings.

The natural stratification, however, exists --- the one by symplectic leaves. It is finite due to Kaledin~\cite{Kal06}, and dimensions of all strata are even (because they are symplectic), so at least at a first glance, it suits the requirements for existence of the category of coherent sheaves, perverse with respect to this stratification (note that for the nilpotent cone this stratification coincides with the one by group orbits).

Symplectic leaves are orbits not of an algebraic group, but rather of a Lie algebroid (of Hamiltonian vector fields). Thus, it seems natural first to develop the theory of perverse coherent modules over Lie algebroids. In order to also account existing group equivariance, we actually need a more general notion of Harish-Chandra (HC) Lie algebroid. 

We now come to the point of stating the main results of the present paper.

\subsection{Main results}
\subsubsection{Existence of perverse t-structure for modules over HC Lie algebroid} \label{intro subsubsec: perverse over Lie algebroids}
Although the main object of our interest is the category of sheaves on symplectic singularities, equivariant with respect to a particular Lie algebroid (we describe it below), we start with a general setting of modules over an arbitrary Harish-Chandra Lie algebroid.

Namely, we assume that $X$ is a variety, $(G, \cL)$ is a HC Lie algebroid on $X$, and we are interested in the category of $\cO_X$-coherent $(G, \cL)$-modules $\Coh^{(G, \cL)} X$, and the corresponding bounded derived category of quasi-coherent sheaves with coherent cohomology $D^b_{\coh} \qcoh^{(G, \cL)} X$ (see Section~\ref{section: Harish-Chandra Lie algebroids} for definitions). 
As in the group-equivariant case~\cite{Bez00}, we assume that there is a finite number of $(G, \cL)$-orbits on $X$, and that dimensions of adjacent (meaning one lies in the closure of the other) orbits differ by at least 2. Moreover, we assume that there is a dualizing object in the derived category $D^b_{\coh} \qcoh^{(G, \cL)} X$ (its existence will be clear in the examples we are interested in). 

Then we have a complete analog of the main results of~\cite{Bez00, AB10}, which guarantees the existence of perverse coherent t-structure under  the above assumptions. Moreover, there is the IC-extension functor from each $(G, \cL)$-orbit, and we get a classification of simple modules in the heart of perverse t-structure as IC-extensions of simple objects on an orbit. The resulting abelian category is Artinian and Noetherian. This is proved in Section~\ref{sec: perverse coherent modules over HC Lie algebroids}. All main results of this section are summarized in Theorem~\ref{main theorem on perverse coherent sheaves}.

We should note that these results are not hard, since the proofs are in many ways parallel to the group-equivariant case of~\cite{Bez00, AB10}. There are some modifications though, the main one being the following: unlike the group-equivariant case, it is not true (even for the case $X = \pt$, which is well-known) for Lie algebroids equivariance that any quasi-coherent equivariant sheaf is a union of its coherent equivariant subsheaves. 
Hence, the following two useful features used in~\cite{Bez00, AB10} do not hold in our situation. First, there is no equivalence of triangulated categories $D^b \Coh^{(G, \cL)} X$ and $D^b_{\coh} \qcoh^{(G, \cL)} X$ (the first category being ``smaller'', see Example~\ref{example: g-modules on point}). We work with the second category for technical reasons. 
Second, it is not true (or at least the standard proof does not apply) that any coherent equivariant sheaf on an open subscheme has some coherent extension to the whole scheme. We get around this difficulty by using the codimension 2 assumption for orbits from the very beginning (then the usual non-derived pushforward provides the desired extension, see Lemma~\ref{serre quotient by subcategory of sheaves supported on closed subset}). So whenever either of these facts is used in~\cite{Bez00, AB10}, we find a detour. 

\subsubsection{Definition of the category for symplectic singularities} \label{intro subsubsec: definition of the category}

Our main results concern symplectic singularities. We propose a definition of the category of perverse coherent (or perverse Poisson) sheaves on an arbitrary symplectic singularity (not necessarily admitting a symplectic resolution). This is done in Section~\ref{Section: symplectic singularities} of the main body of the text; we now provide a brief sketch.

Let $X$ be a symplectic singularity. We assume it is conical, with conical action denoted $\torus_\hbar \curvearrowright X$. The Poisson structure on $X$ equips the sheaf of K\"ahler differentials $\Om_X$ with structure of a Lie algebroid. Modules over this algebroid are called Poisson sheaves or Poisson modules. Orbits of this algebroid are the symplectic leaves of $X$.

It is the category of $\Om_X$-modules (Poisson sheaves) we suggest to define the perverse t-structure on. More precisely, we consider the triangulated category $D^{b}_{\coh} \qcoh^{\Om_X} X$; the results of~\ref{intro subsubsec: perverse over Lie algebroids} apply to this case, allowing us to define the corresponding category of perverse coherent sheaves $\cP_\coh^{\Om_X} X$.

One may also want to account for the contracting $\torus_\hbar$-action equivariance and consider the category $\cP_\coh^{\torus_\hbar, \Om_X} X$. There is also a group $G$ of Hamiltonian graded automorphisms of $X$, and one may want to impose a condition of integrability of sheaves along $G$, leading to the category $\cP_\coh^{\torus_\hbar \times G, \Om_X} (X)$. We believe all these notions are meaningful.
 
Note that in \cite{Cul10} there was an attempt to ``glue'' categories of Poisson sheaves on each orbit of the nilpotent cone into a single category  with perverse t-structure. This is not what we do: restricting to a symplectic leaf $S$, the category we get is the category of modules over $\Om_X \vert_S$, which is not isomorphic to $\Om_S$. For example, on the closed leaf $\{0\} \in X$, the category that arises in our construction is the category of modules over the Lie algebra $\Om_X\vert_0 = T^*_0 X$, and not the category of Poisson sheaves on a point, which is just $\Vect_\bC$.

It is not immediately clear why the particular notion we suggest is reasonable. We justify it by studying it in some examples, see~\ref{intro subsubsec: examples}.

\subsubsection{Simple modules on a symplectic leaf}

As explained above, the simple objects in $\cP_\coh^{\torus_\hbar, \Om_X}(X)$ are the IC-extensions of (cohomologically shifted) simple modules on a symplectic leaf of $X$. So, the problem of describing simple modules on a leaf arises.

Note that in the group-equivariant case of~\cite{Bez00}, this question is completely elementary: simple $G$-equivariant coherent sheaves on a $G$-orbit are the same as irreducible representations of the stabilizer of a point on that orbit, which is reasonably explicit and computable.

In contrast, for our Lie algebroid case, no similar elementary argument seems possible. As explained above, for a symplectic leaf $S \subset X$, the category we are trying to describe is that of $\Om_X \vert_S$-modules on $S$. We were unable to give a complete description of this category in full generality. However, we obtain some partial results, which we now explain.

More generally, one can ask how to describe the category of $\cO$-coherent modules over a transitive Lie algebroid. The first observation is that such a category is Tannakian, hence equivalent to the category of representations of some pro-algebraic group, see Section~\ref{subsection differential galois group}.

The tautological example of a transitive Lie algebroid on a smooth variety $Y$ is the tangent algebroid $\cT_Y$. The category of $\cO$-coherent modules over it (D-modules) is known to be hard to describe in general in the setting of algebraic geometry. This is in contrast with smooth or analytic settings, where it is well-known to be equivalent to the category of representations of fundamental group $\pi_1(Y)$.

So, as a first step, we consider a locally free transitive Lie algebroid in the holomorphic setting, and give a description of the category of $\cO$-coherent modules over it, see Section~\ref{subsec: modules over transitive is equivariantization}, with the main result being Theorem~\ref{thm: L-modules is equivariantization}.
It turns out, the answer in general involves not only $\pi_1(Y)$, but the homotopy groupoid $\pi_{\leq 2}(Y)$ and the inertia bundle of the Lie algebroid. Our method for dealing with this question involves higher categories.

As a second step, we consider the analytification functor, and prove that, in the case of graded modules (i.e. modules over the HC-pair $(\torus_\hbar, \Om_X|_S)$), it is fully faithful. This can be thought of as the fact that all $(\torus_\hbar, \Om_X|_S)$-modules on $S$ have regular singularities (theory of regular singularities for modules over Lie algebroids was developed in~\cite{Kae98}). As often in the theory of regular singularities, we reduce it to the case of projective variety and use GAGA. This is done in Proposition~\ref{prop: gaga for open of codim 2} and in the proof of Theorem~\ref{thm: on simple modules on a symplectic leaf}.

This detour through holomorphic setting allows us to provide the desired estimation of the category of interest, see Theorem~\ref{thm: on simple modules on a symplectic leaf}.

\subsubsection{Examples} \label{intro subsubsec: examples}
Let $\cN$ be the nilpotent cone of a semisimple Lie algebra $\g$, and let $G$ be the corresponding simply connected group. Then there is the category $\cP_\coh^{G \times \torus_\hbar}(\cN)$ of perverse coherent $G \times \torus_\hbar$-equivariant sheaves on $\cN$, as defined in~\cite{Bez00, AB10}.
On the other hand, we have the category $\cP_\coh^{\torus_\hbar, \Om}(\cN)$, suggested in this paper.

These categories are different; however, they turn out to be closely related. Namely, the action of $G$ on $\cN$ is Hamiltonian, meaning that the action of $\g$ factors through $\Om_\cN$. Thus we have the restriction functor $\Coh^{\torus_\hbar, \Om}(\cN) \rightarrow \Coh^{\torus_\hbar, \g}(\cN) \simeq \Coh^{G \times \torus_\hbar}(\cN)$. We show that the corresponding derived functor is t-exact with respect to the perverse t-structures; moreover, it commutes with IC-extension from any orbit; moreover, it defines a bijection between the classes of simple objects. Hence, in particular, there is an isomorphism of the Grothendieck groups of these categories, which also preserves perverse bases:
\begin{equation*}
K^{G \times \torus_\hbar}(\cN) \simeq K^{\torus_\hbar, \Om}(\cN),
\end{equation*}
see Section~\ref{subsec: nilcone} for details.
There are also variants of this result in case $G$ is not necessarily simply connected, and also not including the contracting action, see Theorem~\ref{perverse bases for nilcone are the same}.

Another example of perverse coherent category studied before is $\cP_\coh^{G(\cO) \rtimes \torus_\hbar}(\ol \Gr^\la_G)$, as well as the colimit $\cP_\coh^{G(\cO) \rtimes \torus_\hbar}(\Gr_G)$ (the notations are standard, see Section~\ref{subsec: affine grassmannian slices}).

Let us assume $G$ is simply connected. Then, in particular, there is an open conical symplectic singularity in $\ol \Gr^\la_G$ --- the transversal slice to the zero orbit, denoted $\cW^\la_0$. The action of the Lie algebra $\g[t]$ on $\ol \Gr^\la$ restricts to the open $\cW^\la_0$, and we show that it factors through $\Om_{\cW^\la_0}$.
This allows us to connect our category with previously studied objects. In particular, we show that the restriction functor induces an inclusion of Grothendieck groups
\begin{equation*}
K^{\torus_\hbar, \Om}(\cW^\la_0) \hookrightarrow K^{G(\cO) \rtimes \torus_\hbar} (\ol \Gr^\la_G),
\end{equation*}
which maps perverse basis to (a part of) perverse basis. Moreover, in the colimit $\la \rightarrow \infty$, these two K-groups are isomorphic, and the isomorphism respects perverse bases. This is shown in Section~\ref{subsec: affine grassmannian slices}.

So, the construction we propose in the present paper generalizes both previously studied perverse coherent bases. Note that in both these cases this basis is known or expected to be canonical --- in the sense of Lusztig or Kazhdan--Lusztig, see~\ref{intro subsec: perverse coherent start}. We expect that for other symplectic singularities, our basis should share similar properties, to some extent. This is partially confirmed by our studies of the basis for Slodowy slices and affine Grassmannian slices to a nonzero orbit, see Section~\ref{subsec: other examples}.

\subsection{Directions for further research} \label{intro subsubsec: further questions}
Below, we briefly list a few possible directions for further research.

\subsubsection{More examples}
It would be interesting to investigate properties of perverse coherent basis for other examples of symplectic singularities, such as hypertoric varieties or quiver varieties.

\subsubsection{Lifting to symplectic resolutions}
Let $\wti \cN \rightarrow \cN$ be the Springer resolution. In this case, there is a natural way to ``lift'' the perverse t-structure from $D^b \Coh^G(\cN)$ to $D^b \Coh^G(\wti \cN)$: namely, there is the noncommutative Springer resolution $A$ \cite{Bez06b}, which can be considered as a sheaf of algebras on $\cN$, and one can define the category of perverse $A$-modules. This is a t-structure on $D^b \Coh^G(\wti \cN)$, called {\it perversely-exotic} and the basis of simple objects in it can be identified with the KL canonical basis in the anti-spherical module for the affine Hecke algebra of $G^\vee$, see~\cite[6.2]{BM13}.

Noncommutative resolutions exist for other symplectic resolutions, see~\cite{Kal08}. It would be very interesting to ``lift'' the basis or the t-structure we proposed for a symplectic singularity, to a symplectic resolution, whenever it exists.

\subsubsection{Comparison with bases in K-theory}
Classes of simple perverse coherent sheaves on a symplectic singularity $X$ form a basis in the Grothendieck group of Poisson sheaves on $X$.
As we explain in Section~\ref{sec: examples}, in some examples, this basis is actually related to a basis in equivariant K-theory of (possibly different!) variety. There are known examples of bases in equivariant K-theory, such as K-theoretic stable envelopes \cite{MO12, OS22} and Hikita's canonical basis~\cite{Hik20}. It would be interesting to investigate the relation between these constructions and ours.

\subsubsection{Restriction to Lagrangians}
Koppensteiner \cite{Kop15} proved that $\cF \in D^b(\Coh^G X)$ is perverse if and only if $i_Z^! \cF$ is concentrated in a single cohomological degree for sufficiently many \textit{measuring subvarieties} $Z$ of $X$. We expect that the same argument should also apply to our notion of perversity. For symplectic singularities, a natural choice of a measuring subvariety would be a Lagrangian subvariety. Lagrangian subvarieties may arise as supports of holonomic modules over quantizations (see \cite{Los17}). It would be interesting to investigate this further.

\subsubsection{Bimodules over quantizations}
Poisson sheaves, which we work with, are the semi-classical variant of Harish-Chandra bimodules over quantizations. It would be interesting to study a quantum analog of the notion we propose here, possibly in the modular setting. See~\cite[Section~4.5]{Los23},~\cite{Los21}.

\subsection{The paper is organized as follows}
In Section~\ref{section: Harish-Chandra Lie algebroids}, we collect all the required facts about HC Lie algebroids and modules over them. While some results are standard, others have not appeared in the literature to the best of our knowledge, and may be of independent interest. 

In Section~\ref{sec: perverse coherent modules over HC Lie algebroids}, we construct the perverse t-structure for modules over HC Lie algebroids, define the IC-extension functor, and describe simple objects in the heart of the perverse t-structure. The main results of this Section are summed up in Theorem~\ref{main theorem on perverse coherent sheaves}. This Section largely follows the papers~\cite{Bez00, AB10}, adapting them to our setting.

In Section~\ref{Section: symplectic singularities}, we recall the required properties of conical symplectic singularities, and propose a definition of the category of perverse coherent sheaves on them, Definition~\ref{def: category for sympl sing}. We then study simple modules on a symplectic leaf, Theorem~\ref{thm: on simple modules on a symplectic leaf}.

In Section~\ref{sec: examples}, we examine the suggested general notion in particular examples. We first study the case of the nilpotent cone and prove that our basis coincides with the one of~\cite{Bez00}, Theorem~\ref{perverse bases for nilcone are the same}. We then turn to the case of affine Grassmannian slice to the zero orbit and relate our basis to the known one, Theorem~\ref{thm: basis for slice in affine gr}. Finally, we speculate about other examples in Sections~\ref{subsec: other examples} and~\ref{subsec: double affine grass}.

\subsection*{Acknowledgements}
None of what appears in this paper would have been possible without the patient guidance of Roman Bezrukavnikov. It is a pleasure to thank him. 

Different parts of the paper also owe their existence to discussions with many mathematicians. I am particularly indebted to the following people for help with the following parts of the paper: Alexandra Utiralova for Section~\ref{subsection differential galois group}; Pavel Etingof, Alexander Petrov, and Ekaterina Bogdanova for Section~\ref{subsec: modules over transitive is equivariantization}; Ivan Losev and Andrei Ionov for Section~\ref{subsec: gaga}; Dmytro Matvieievskyi for Section~\ref{subsec: symplectic singularities}; Vasily Krylov for Lemma~\ref{lem: centralizer acts faithfully on slice}; Dinakar Muthiah for Section~\ref{subsec: double affine grass}.

I also thank  Ivan Losev, Michael Finkelberg, and Vasily Krylov for reading a draft of this article and helping improve the presentation.

\section{Lie algebroids} \label{section: Harish-Chandra Lie algebroids}

In this section, we collect results about Harish-Chandra Lie algebroids and modules over them, required for the purposes of the present article. 
For a more thorough overview, see \cite{Kae98, BB93} for the more relevant to us algebraic case, or \cite{Mac05, Mei17} for the better studied smooth and analytic cases.

\subsection{Lie algebroids}
Let $X$ be a connected scheme of finite type over the field $\bC$ of complex numbers.
Unless otherwise specified, we do not assume that $X$ is smooth. The tangent sheaf $\cT_X$ is defined as the dual to the sheaf of K\"ahler differentials $\Om_X$.

\begin{definition}
A \textbf{Lie algebroid} $(\cL, \rho)$ on $X$ is a quasi-coherent sheaf $\cL$, equipped with a Lie bracket $[\cdot, \cdot]: \cL \otimes_{\bC} \cL \rightarrow \cL$ and an $\cO_X$-linear anchor map $\rho: \cL \rightarrow \cT_X$ to the tangent sheaf, such that $\rho$ intertwines the Lie brackets, and for any local sections $\ell_1, \ell_2 \in \cL(U)$, $f \in \cO_X(U)$, we have $[\ell_1, f \ell_2] = f[\ell_1, \ell_2] + \rho(\ell_1)(f) \ell_2$.
\end{definition}

For any $\cL$, the kernel of the anchor map $\h := \ker \rho \subset \cL$ is an $\cO_X$-linear sheaf of Lie algebras. We call $\h$ the {\bf inertia sheaf}.

Given $\cL$, one can form the universal enveloping sheaf of algebras $\cU(\cL)$ in the obvious way. See \cite{BB93} for details, where algebras of the form $\cU(\cL)$ are called the {\bf D-algebras}.

\begin{definition}
An $\cL$\textbf{-orbit} on X is a maximal locally closed connected subscheme $S \subset X$ such that for any point $s \in S$, the image of $\rho_s$ is equal to $T_s S \subset T_s X$.
\end{definition}

Orbits do not necessarily exist in general.
An algebroid $\cL$ is called {\bf transitive} if the anchor map $\rho$ is surjective (equivalently, if $X$ is the only orbit of $\cL$).

We call a locally closed subvariety $Y \subset X$ {\bf $\cL$-invariant} if at any $y \in Y$, $\im \rho_y \subset T_y Y \subset T_y X$.

\begin{lem}
Any orbit of a Lie algebroid $\cL$ is a smooth variety.
\end{lem}

\begin{proof}
On an orbit $S$, the function $\dim T_s S$ is upper semicontinuous, while $\dim  (\im(\rho_s) )$ is lower semicontinuous. Since they coincide, $\dim T_s S$ is constant, and hence $S$ is smooth.
\end{proof}

\begin{definition} \label{def: module over Lie algebroid}
A {\bf module} over a Lie algebroid $\cL$ is a quasi-coherent sheaf $M$ together with a morphism $\cL \rightarrow \End_\bC (M)$ such that $\ell(fm) = \rho(\ell)(f) m + (f\ell)m$, $(f\ell)m = f(\ell m)$ for local sections $\ell \in \cL(U), f \in \cO_X(U), m \in M(U)$.
\end{definition}

Let us give a few examples of Lie algebroids relevant to the present paper:
\begin{example} \label{example of lie algeboroids}
\begin{enumerate}[a)]
	\item  The tangent sheaf $\cT_X$ is tautologically a Lie algebroid on $X$. In case $X$ is smooth, $\cT_X$ is transitive, $\cT_X$-modules are called D-modules, and $\cU(\cT_X)$ is the sheaf of differential operators on $X$.
	\item An action of a Lie algebra $\g$ on $X$ is, by definition, a Lie algebroid structure on the trivial sheaf of Lie algebras $\cO_X \otimes \g$. If an algebraic group $G$ acts on $X$, it induces the action of its Lie algebra $\g$ on $X$. Any $G$-equivariant sheaf is automatically an $\cO_X \T \g$-module. If $G$ is connected, $\cO_X \otimes \g$-orbits coincide with $G$-orbits.

    	\item \label{example: cotangent algebroid} A Poisson structure on $X$ naturally equips the sheaf of K\"ahler differentials $\Om_X$ with the structure of Lie algebroid. The Lie bracket is determined by $[df, dg] = d \{f, g\}$ (locally); the anchor map $\Om_X \rightarrow \cT_X$ is given by the Poisson bivector. The orbits of this algebroid are called the symplectic leaves. Modules over this Lie algebroid are called Poisson sheaves or Poisson modules. See \cite{Pol97} for details on this example.
    We call it the Poisson Lie algebroid.

    Note that if $X$ is smooth and the Poisson structure comes from a symplectic structure, then the anchor map $\Om_X \xrightarrow{\sim} \cT_X$ is an isomorphism; in particular, Poisson sheaves are just D-modules (and this notion does not depend on the symplectic structure).
\end{enumerate}
\end{example}

To clarify the notion of module over a Lie algebroid, we suggest the following example:
\begin{example} \label{example: adjoint modules}
Let $(\cL, \rho)$ be a Lie algebroid on $X$.
\begin{enumerate}[a)]
    \item The adjoint action of $\cL$ on itself does \emph{not} define an $\cL$-module, because it is not $\cO_X$-linear.
    \item \label{example: adjoint action on ker rho} However, it is elementary to see that the adjoint action of $\cL$ on the inertia sheaf $\h = \ker \rho \subset \cL$ does define an $\cL$-module.
\end{enumerate}
\end{example}

When we say that an $\cL$-module is coherent, we always mean that it is $\cO_X$-coherent (as opposed to being coherent over the universal enveloping sheaf of algebras $\cU(\cL)$).

\begin{lem} \label{coherent module is locally free}
A coherent module over a transitive Lie algebroid is locally free.
\end{lem}
\begin{proof}
This statement is well known for the case of D-modules. The proof of the general case is identical to the proof for D-modules given in \cite[Theorem 1.4.10]{HT07}. We provide the full proof for the reader's convenience.

Let $\cL$ be a transitive Lie algebroid on $X$, and $M$ be a coherent module over it.

Let $x \in X$ be a closed point, and consider the stalk $M_x$, which is a module over the local ring $\cO_x$. Let $\fm_x \subset \cO_x$ be the maximal ideal, and $\bar s_1, \hdots, \bar s_n$ be a basis of $M_x / \fm_x M_x$. Lift these to elements $s_1, \hdots, s_n \in M_x$. By Nakayama's lemma, $s_1, \hdots, s_n$ generate $M_x$. We have to show that they are linearly independent over $\cO_x$.

Assume, for contradiction, that we have $\sum_{i = 1}^n \phi_i s_i = 0$ for $\phi_i \in \cO_x$. Define the function $\ord_x$ by setting $\ord_x \phi = n$ if $\phi \in \fm_x^n$ but $\phi \notin \fm_x^{n + 1}$. Let $\nu = \min_i (\ord_x \phi_i)$. We may assume $\nu = \ord_x \phi_1$. As $\bar s_i$ are linearly independent, we have $\nu \geq 1$. Assume $\nu$ takes the minimal possible value among all choices of $\phi_i$.

It is clear that there exists a vector field $\mu$, defined locally around $x$, such that $\mu(\phi_1) \neq 0$ and $\ord_x \mu(\phi_1) < \nu$.

Now consider $\cL_x = \cL \otimes \cO_x$ and $\rho_x: \cL_x \rightarrow \Der (\cO_x)$. 
Since $\cL$ is transitive, $\rho_x$ is surjective, so there is $\kappa \in \cL_x$ such that $\rho_x(\kappa) = \mu$. Then we have:
\[
0 = \kappa (\sum_{i = 1}^n \phi_i s_i) = \sum_{i = 1}^n \mu(\phi_i) s_i + \sum_{i = 1}^n \phi_i \kappa( s_i).
\]
Since $s_i$ generate $M_x$, we have $\kappa(s_i) = \sum_j a_{ij} s_j$ for some $a_{ij} \in \cO_x$. Thus,
\[
0 = \sum_i (\mu(\phi_i) + \sum_j \phi_j a_{ji})s_i.
\]
The coefficient of $s_1$ is $\mu(\phi_1) + \sum_j \phi_j a_{j1}$. We have $\ord_x \mu(\phi_1) < \nu$ and $\ord_x(\sum_j \phi_j a_{ij}) \geq \nu$. Therefore, the entire coefficient has order strictly less than $\nu$. This is a contradiction.
\end{proof}
In particular, the above proposition tells that for a general Lie algebroid $\cL$, coherent $\cL$-modules are ``smooth along the stratification by $\cL$-orbits'', which will be important for the construction of perverse coherent t-structure later.

\subsection{Harish-Chandra Lie algebroids} \label{subseq: HC lie algebroids}

Let $G$ be an algebraic group acting on $X$, and let $\g$ be the Lie algebra of $G$.

\begin{definition}
A {\bf Harish-Chandra (HC) Lie algebroid} $(G, \cL)$ is a Lie algebroid $\cL$, equipped with a $G$-equivariant structure, and a $G$-equivariant morphism $i: \cO_X \T \g \rightarrow \cL$ of Lie algebras, such that the Lie bracket and the anchor map of $\cL$ are $G$-equivariant, and, moreover, two natural actions of $\g$ on $\cL$ --- one coming from $G$-equivariance and the other from $i$ --- coincide.
\end{definition}

One naturally defines the notion of $(G, \cL)$-module. Throughout the paper, when we refer to a ``$(G, \cL)$-module'', we always mean {\bf strongly-equivariant} module, see \cite[Section~1.8]{BB93}.

$(G, \cL)$-orbits are, by definition, the $\cL$-orbits. Lemma~\ref{coherent module is locally free} guarantees that any $\cO$-coherent module over a transitive Harish-Chandra Lie algebroid is locally free, since it is a module over the underlying Lie algebroid.

We denote by $\Coh^{(G, \cL)} X$ the category of $\cO_X$-coherent $(G, \cL)$-modules. This category is abelian, and one naturally defines the notion of kernels, cokernels, direct sum, and tensor product (over $\cO_X$) in it.

We also denote by $\qcoh^{(G, \cL)} X$ the category of quasi-coherent $(G, \cL)$-modules.

\subsection{Differential Galois group} \label{subsection differential galois group}

It was observed by Katz \cite{Kat72, Kat82, Kat87} that the category of $\cO$-coherent D-modules is Tannakian in the sense of Deligne--Milne \cite{DM82}. In fact, the same holds for modules over an arbitrary transitive HC Lie algebroid.

Assume that $(G, \cL)$ is a transitive HC Lie algebroid on $X$. Consider the tensor category $\Coh^{(G, \cL)}(X)$ of coherent $(G, \cL)$-modules. From Lemma~\ref{coherent module is locally free}, it is evident that the functor $\cF \mapsto \shhom_{\cO_X}(\cF, \cO_X)$ endows this category  with a rigid structure.
Moreover, take any closed $x \in X$, and consider the functor $F_x: \Coh^{(G, \cL)}(X) \rightarrow \Vect$, $M \mapsto M_x$, which maps $M$ to its fiber at $x$. Recall that $X$ is connected.


\begin{lem}
$\Coh^{(G, \cL)}(X)$ is a Tannakian category with $F_x$ being a fiber functor.
\end{lem}
\begin{proof}
The only nontrivial parts are exactness and faithfulness of $F_x$, which we now establish. First, note that $F_x$ factors through the functor, forgetting the $G$-action $\Coh^{(G, \cL)}(X) \rightarrow \Coh^\cL(X)$, which is exact and faithful, so it suffices to verify. our claims for $\Coh^\cL(X)$.

To show exactness, first consider the completion functor $\Coh^\cL(X) \rightarrow \Coh^{\cL^{\wedge x}}(X^{\wedge x})$, which is exact on coherent sheaves. For $X^{\wedge x}$, it was shown in \cite[Theorem~A.7.3]{Kap07} that $F_x$ factors through an abelian equivalence $\Coh^{\cL^{\wedge x}}(X^{\wedge x}) \simeq (\ker \rho)_x\mods$, and hence is exact.

To establish faithfulness, we show that for any $\cF \in \Coh^\cL(X)$, the natural map $\Gamma(\cF^\cL) \rightarrow \cF_x$ is injective (here $\cF^\cL$ denotes the subsheaf of $\cL$-invariants, so $\Gamma(\cF^\cL)$ stands for global ``flat'' sections). 
Applying this to $\cF = \shhom_{\cO_X}(M, N)$ we get the faithfulness of $F_x$, since 
\[\Hom_\cL(M, N) = \Gamma((\shhom_{\cO_X}(M, N))^\cL).\]
To prove injectivity, we reduce it to the case of $\cT_X$-modules, where it is well known. Note that for a module $\cF$ over an algebroid $(\cL, \rho)$, the subsheaf $\cF^{\ker \rho}$ is naturally a $\cT_X$-module. Thus, we have
\begin{equation*}
\Gamma(\cF^\cL) = \Gamma(((\cF)^{\ker \rho})^{\cT_X}) \hookrightarrow (\cF^{\ker \rho})_x \hookrightarrow \cF_x,
\end{equation*}
where the last embedding follows from the exactness of $F_x$ shown above, applied to the embedding of $\cL$-modules $\cF^{\ker \rho} \hookrightarrow \cF$.
\end{proof}

We denote by $\textrm {Gal}^{(G, \cL)}(X)$ the Tannakian group of the category $\Coh^{(G, \cL)}(X)$. This is a pro-algebraic group, for which there is a canonical equivalence of tensor categories
\[
\Rep \Gal^{(G, \cL)}(X) \simeq \Coh^{(G, \cL)}(X).
\]
For a smooth $X$ and $\cL = \cT_X$, the group $\Gal^{\cT_X}(X)$ is what is sometimes called the \textbf{differential Galois group} of $X$ (hence our notation). In this case $ \Coh^{\cT_X}(X)$ is the category of algebraic local systems on~$X$.

\subsection{Direct and inverse images} \label{subsec: direct and inverse image}
For a general morphism of varieties $\pi: Y \rightarrow X$ and a Lie algebroid $(\cL, \rho)$ on $X$, define
\[
\pi^+ \cL = \pi^* \cL \times_{\pi^* \cT_X} \cT_Y,
\]
locally given by sections 
\[
\{(\ell, v) \in \pi^* \cL \oplus \cT_Y \vert (\pi^* \rho) (\ell) = d\pi(v)  \}.
\]
It is a Lie algebroid on $Y$ with a natural bracket and anchor map, see \cite[2.4.5]{Kae98} for details.

In \cite[3.4 -- 3.5]{Kae98}, the functors of inverse and direct images are defined between categories $\qcoh^{\cL}(X)$ and $\qcoh^{\pi^+ \cL}(Y)$. Note that, at the level of sheaves, 
the direct image $\pi_+$ does not coincide with the usual sheaf-theoretic pushforward, 
but rather involves an intermediate sheaf $\cD_{Y \rightarrow X}$, in style of a similar definition for D-modules:
\begin{equation*}
\pi_+ M = \pi_*(\pi^*\cU(\cL) \otimes_{\cU(\pi^+ \cL)} M).
\end{equation*}

Note that in fact, in \cite{Kae98} one deals only with derived versions of these functors, since, similarly to D-modules, non-derived direct image is not well behaved in general. We write $\pi_+$ in the non-derived sense, since our goal in Lemmata~\ref{lem: pushforward for open embedding},~\ref{direct and inverse images for locally closed embedding} below is to show that in certain good situations, non-derived versions are also nice.

First, we note that in the case of open embedding, $\pi_+$ agrees with the usual quasi-coherent pushforward $\pi_*$ (more formally, pushforward commutes with the functor that forgets the Lie algebroid action):
\begin{lem} \label{lem: pushforward for open embedding}
Suppose $\pi: Y \rightarrow X$ is an open embedding, and $\cL$ is a Lie algebroid on $X$. Then $\pi^* \cL \simeq \pi^+ \cL$. 

If $N$ is a $\pi^*\cL$-module on $Y$, then $\pi_* N \simeq \pi_+ N$ as sheaves. 
\end{lem}

\begin{proof}
The first claim follows from the isomorphism $\pi^* \cT_X \simeq \cT_Y$. The second claim follows from the isomorphism $\cD_{Y \rightarrow X} = \pi^*(\cU(\cL)) = \cU(\pi^* \cL)$, since it implies $\pi_+ N \simeq \pi_*(N \T_{\pi^* \cU(\cL)} \cU(\pi^* \cL)) \simeq \pi_* N$. 
\end{proof}

Next, we show that the things simplify when $Z \rightarrow X$ is a locally closed inclusion of an $\cL$-invariant subvariety:
\begin{lem} \label{direct and inverse images for locally closed embedding}
Let $i_Z: Z \rightarrow X$ be a locally closed embedding of an $\cL$-invariant subscheme. 
\begin{enumerate}[a)]
    \item \label{direct and inverse a)} $i_Z^* \cL$ has a natural structure of a Lie algebroid on $Z$ and there is a canonical isomorphism $i_Z^*\cL \simeq i^+ \cL$.

    \item \label{direct and inverse b)} For an $\cL$-module $M$, the pullpack $i^* M$ is naturally a $i_Z^*\cL$-module. This agrees with the definition in \cite[Section 3.4]{Kae98}: under the identification $i_Z^*\cL = i_Z^+\cL$, we have $i_Z^*M \simeq i_Z^+ M$ canonically for any $\cL$-module $M$.

    \item \label{direct and inverse c)} For an $i^*\cL$-module $N$, the pushforward $(i_Z)_* N$ is naturally a $\cL$-module. This agrees with the definition in \cite[Section 3.5]{Kae98}: under identification $i_Z^*\cL = i_Z^+\cL$, we have $(i_Z)_*N \simeq i_+ N$ canonically for any $i^* \cL$-module $N$.
\end{enumerate}
In particular, $(i_Z^*, (i_Z)_*)$ is an adjoint pair of functors between categories $\qcoh^\cL X$ and $\qcoh^{i_Z^* \cL} Z$.
\end{lem}
Note that in \cite{Kae98}, the functor $i_+$ is defined only in the derived setting and under the assumption of $X, Z$ being smooth; in our case of consideration (locally closed inclusion of an invariant subscheme), neither of these requirements is necessary.
\begin{proof}
Part \ref{direct and inverse a)} follows essentially from the definitions. We have an embedding $\cT_Z \subset i_Z^* \cT_X$, and the image of $i_Z^* \rho$ lands in $\cT_Z$. Hence $(i_Z^*\cL, i_Z^* \rho)$ is an algebroid on $Z$. For the same reason, local sections of $i^+ \cL$ have the form $(\ell, i_Z^*\rho(\ell))$, where $\ell$ is a local section of $i_Z^* \cL$, so we get a canonical identification $i_Z^* \cL \simeq i_Z^+ \cL$.

Part \ref{direct and inverse b)} is immediate.

For part \ref{direct and inverse c)}, using that $Z$ is $\cL$-invariant, one can check that the natural surjection $i_Z^* \cU(\cL) \rightarrow \mathcal{U}(i_Z^* \cL)$ of algebras on $Z$ is actually an isomorphism (this can be easily verified separately for open and closed embeddings).
Hence, $(i_Z)_+ N \simeq (i_Z)_*(N \T_{U(i^* \cL)} U(i^* \cL)) \simeq (i_Z)_* N$, and this is compatible with the natural algebroid actions.
\end{proof}

It is straightforward to see that the above lemma generalizes to the case of a HC algebroid $(G, \cL)$ (the functors naturally respect the group equivariance), and we get adjoint functors between the categories $\qcoh^{(G, \cL)} X$ and $\qcoh^{(G, i_Z^* \cL)} Z$, for a $(G, \cL)$-invariant locally closed subscheme $Z$ of $X$.

Abusing notations, when the context is clear, we write $\cL$ instead of $i_Z^* \cL$ to denote the restricted algebroid on $Z$.


We note that the functor $i_Z^*: \qcoh^{(G, \cL)}(X) \rightarrow \qcoh^{(G, \cL)}(Z)$ is right-exact and preserves $\cO$-coherence.

The functor $i_{Z*}: \qcoh^{(G, \cL)} Z \rightarrow\qcoh^{(G, \cL)} X$ is left-exact and preserves coherence in case $i_Z$ being a closed embedding.

We also define the functor $i_Z^!: \qcoh^{(G, \cL)}(X) \rightarrow \qcoh^{(G, \cL)}(Z)$ as follows. For $i_Z$ being an open embedding, it coincides with $i_Z^*$; for $i_Z$ being a closed embedding, it is defined as $\Hom_{\cO_X} (\cO_Z, -)$ (note that $Z$ is $\cL$-invariant here, so $\cO_Z$ is naturally a $(G, \cL)$-module); for $i_Z$ being locally closed, it is defined as the composition of these (exact and left-exact) functors. This functor also preserves coherence.

\subsection{Modules over transitive Lie algebroid in holomorphic setting} \label{subsec: modules over transitive is equivariantization}

Now let $X$ be a complex analytic space. All notions and results of previous subsections have straightforward analogs in this situation (see \cite{Kae98}, where these set-ups are treated in parallel).

For this subsection, we assume $X$ is a connected smooth complex manifold, and $\cL$ is a locally free transitive Lie algebroid of finite rank on $X$. Our goal here is to describe the category $\Coh^{\cL}(X)$ of $\cO_X$-coherent (equivalently, locally free) $\cL$-modules in this setting.

Let $\h = \ker \rho$ be the inertia sheaf. Then $\h$ is a locally trivial bundle of $\cO_X$-linear Lie algebras by \cite[Proposition~3.6]{Mei21}. In particular, the fibers $\h_x$ are isomorphic as Lie algebras for all points $x \in X$.

\subsubsection{$\cL$-modules as a local system of categories}
Here and throughout this section, by ``category'' we always mean ``1-category''.
Recall the notion of a {\it local system of categories}. For the $\infty$-setting, see \cite[Appendix~A]{Lur17}. We make things as explicit as possible for the relevant to us case of 1-categories.
\begin{definition}
Fix a good open cover $\{U_i\}$ of $X$.
The \textbf{local system of categories} on $X$ is the following data: 
\begin{itemize}
\item A category $\cC_i$, assigned to each open $U_i \subset X$; 
\item An equivalence $F_{ij}: \cC_i \rightarrow \cC_j$ for each double intersection $U_i \cap U_j$;
\item An invertible natural transformation $\phi_{ijk}: F_{ij} \circ F_{jk} \xrightarrow{\sim} F_{ik}$ for every triple intersection $U_i \cap U_j \cap U_k$,
\end{itemize}
such that for any quadruple intersection $U_i \cap U_j \cap U_k \cap U_\ell$, the two natural transformations $\phi_{ij\ell} \circ (\id \times \phi_{jk\ell}),~ \phi_{ik\ell} \circ (\phi_{ijk} \times \id): F_{ij} \circ F_{jk} \circ F_{k \ell}  \rightarrow F_{i\ell}$ are equal.

If all $\cC_i$ are (non-canonically) equivalent to a fixed category $\cC$, we say that this is a \textbf{local system with fiber $\cC$}.

\end{definition}

Given a local system of categories defined on some fixed cover, one easily defines its sections over an arbitrary open $U \subset X$, see \cite[Appendix A]{Lur17}.

Our first observation is the following:
\begin{lem} \label{lem: L-modules is local system of categories}
Consider the sheaf of categories on $X$, that assigns to each open $U \subset X$ the category $\Coh^{\cL\vert_U}(U)$ of $\cL$-modules on $U$. This defines a local system of categories with fiber $(\h_x\mods)$ --- the category of finite-dimensional modules over the Lie algebra $\h_x$.
\end{lem}
\begin{proof}
By \cite[Theorem~8.5.1, Corollary~8.5.5]{DZ05}, in some small neighborhood $U$ of any point $x \in X$, the Lie algebroid $\cL$ decomposes as a direct sum $\cL\vert_U \simeq \cT_U \oplus (\cO_U \T \h_x)$\footnote{We chose to cite \cite{DZ05}, because it is explicitly claimed there that this holds in the holomorphic setting; the proof is the same as in other sources that treat only the case of real manifolds}. 
Then one easily sees that $\Coh^{\cL\vert_U}(U) \simeq \h_x\mods$: indeed, any $\cL\vert_U$-module $V$ is trivialized by means of $\cT_U$-action, and is isomorphic to $\cO_{U} \T V_x$, where $V_x$ is a representation of $\h_x$ (see e.g.~\cite[Theorem~6.5.12]{Mac05} for the case of real manifolds; the claim we make here is very easy to verify for the holomorphic setting as well).

For two $U_1, U_2$ as above, an equivalence between modules over trivializations on $\cL\vert_{U_1 \cap U_2}$ is determined by an isomorphism of trivializations of $\cL$ on $U_1$ and $U_2$. The data of cocycle condition is determined from the gluing data of~$\cL$.
\end{proof}

We denote this local system of categories by $\cL$-$\mathfrak{mod}$. The category $\Coh^{\cL}(X)$ is equivalent to the global sections of this local system.

Thus, it is natural to pose a more general question of describing the global sections of a local system of categories. Note that its 0-categorical analog is familiar: for a local system with fiber $V$ (a vector space), its global sections are isomorphic to $V^{\pi_1(X)}$. What follows is a 1-categorical generalization.

\subsubsection{Higher categories notations and constructions} \label{subsubsec: higher categories}
Below we introduce some notations. All of them are standard in the setting of $\infty$-categories, see e.g. \cite{Lur17}. We elaborate and make them explicit in the 1- and 2-categorical cases required for us.

By $\pi_{\infty}(X)$ we mean the homotopy $\infty$-groupoid of $X$ (also sometimes denoted $\operatorname{Sing} X$ --- the singular simplicial set associated to $X$). We denote by $\pi_{\leq 2} (X)$ its 2-truncation. Explicitly, $\pi_{\leq 2} (X)$ is equivalent to the 2-groupoid with one object (recall that $X$ is connected); its 1-morphisms are elements of the fundamental group of $X$ with natural composition; its 2-morphisms are classes of homotopies between loops. We write $\pi_{\leq 1}(X)$ for the fundamental groupoid and use standard notation $\pi_1(X)$, $\pi_2(X)$, or $\pi_1(X, x), \pi_2(X, x)$ for the corresponding homotopy groups.

For any 1-category $\cC$, there is the monoidal category $\Aut \cC$; equivalently, $\Aut \cC$ is a 2-groupoid with one object, see e.g.~\cite[Example~2.12.6]{EGNO15}; that is how we treat it going forward.

Given a 2-groupoid $\cG$ with one object and a 2-functor $\cG \rightarrow \Aut \cC$, we call this an {\bf action of $\cG$ on $\cC$}. Explicitly, an action assigns to each 1-morphism in $\cG$ an auto-equivalence of $\cC$, and to each 2-morphism in $\cG$, an invertible natural transformation between the corresponding auto-equivalences.

1-categories with $\cG$-action form a 2-category. We describe its 1-morphisms. A 1-morphism between 1-categories $\cC$ and $\cD$ with $\cG$-action is a pair $(F, u)$, where $F: \cC \rightarrow \cD$ is a functor such that for any object $c \in \cC$, the two natural actions on the object $F(c)$, coincide; 
$u$ is a family of invertible natural transformations $u_g: F \circ g \rightarrow g \circ F$ for any 1-morphism $g$ of $\cG$, such that the natural diagram 
\begin{equation} \label{eq: diag for composition of nat transforms}
\begin{tikzcd}
	{F \circ g \circ h} && {g \circ h \circ F} \\
	& {g \circ F \circ h}
	\arrow["{u_{gh}}", from=1-1, to=1-3]
	\arrow["{u_{g} \circ h}"', from=1-1, to=2-2]
	\arrow["{g \circ u_{h}}"', from=2-2, to=1-3]
\end{tikzcd}
\end{equation}
commutes for all $g, h$.

Given an action of $\cG$ on a 1-category $\cC$ as above, one can form the {\it equivariantization} (a.k.a. category of equivariant objects), denoted $\cC^\cG$. It can be defined as the category of 1-morphisms $\Hom_{\cG}(\mathrm{Triv}, \cC)$ in the 2-category of 1-categories with $\cG$-action, described above. Here $\operatorname{Triv}$ is the trivial category with one object. Explicitly, it is described as follows.

$\cC^\cG$ is a 1-category; its objects are of the form $(c, u)$, where $c \in \cC$ is such that the image of the 2-morphisms group $\Hom_{\cG}(\id, \id)$ between identity 1-morphisms of $\cG$ in $\Aut_{\cC}(c)$ (given by the action), is trivial; $u = \{ u_g: c \rightarrow g(c) \}$ is a family of isomorphisms for all 1-morphisms $g$ of $\cG$, with an analog of diagram~\eqref{eq: diag for composition of nat transforms} to hold. The morphisms in $\cC^\cG$ are defined to be commuting with~$\{u_g\}$.
This structure is closely analogous to a more standard equivariantization w.r.t. a (1-)group action, see e.g. \cite[Definition~2.7.2]{EGNO15}).

\subsubsection{Local systems of 1-categories}
Let $\cC$ be a category. 

\begin{lem} \label{lem: local systems are 2-reps of pi_2}
Local systems of categories on $X$ with fiber $\cC$ are in correspondence with actions of $\pi_{\leq 2}(X)$ on $\cC$, that is 2-functors $\pi_{\leq 2} (X) \rightarrow \Aut \cC$.

Moreover, the 2-category of local systems of categories on $X$ is equivalent to the 2-category of categories with a $\pi_{\leq 2}(X)$-action.
\end{lem}
Of course, this lemma should be viewed as a categorical analog of the fact that local systems with fiber $V$ (a vector space) are in correspondence with representations of $\pi_1(X)$ on $V$. Note that as we lift the categorical level, $\pi_2(X)$ also starts playing a role.
\begin{proof}
For $\cC$ being a $\infty$-category, it is proven in \cite[Theorem~A.1.15]{Lur17} that local systems with fiber $\cC$ are in correspondence with $\infty$-functors $\pi_{\infty}(X) \rightarrow \Aut \cC$. When  $\cC$ happens to be a 1-category, the functor factors through the 2-truncation $\pi_{\leq 2}(X)$, and we obtain the claim.
\end{proof}

We now describe the global sections of a local system in the above terms.

\begin{lem} \label{lem: global sections of loc sys is equivariantization}
Let $\mathfrak L$ be a local system of categories with fiber $\cC$ on $X$. Under the correspondence of Lemma~\ref{lem: local systems are 2-reps of pi_2}, the global sections of $\mathfrak L$ are equivalent to the equivariantization: $\Gamma(\mathfrak L) \simeq \cC^{\pi_{\leq 2}(X)}$.
\end{lem}
The 0-categorical analog of this claim is familiar: the global flat sections of a local system with fiber $V$ (vector space) are isomorphic to invariants $V^{\pi_1(X)}$.
\begin{proof}

Taking global sections is the same as considering 1-morphisms from the trivial local system of trivial categories. Hence, we get $\Gamma(\mathfrak L) \simeq \Hom_{\pi_{\leq 2}(X)}(\operatorname{Triv}, \cC) \simeq \cC^{\pi_{\leq 2}(X)}$ (see Section~\ref{subsubsec: higher categories} for definitions and explanations).
\end{proof}

\subsubsection{Main result}
Summing up all of the above and returning to the initial question, we get:
\begin{thm} \label{thm: L-modules is equivariantization}
The category $\Coh^{\cL}(X)$ is equivalent to the equivariantization of the category $\h_x\mods$ under the action of 2-groupoid $\pi_{\leq 2}(X)$, determined by $\cL$:
\begin{equation*}
\Coh^\cL(X) \simeq (\h_x\mods)^{\pi_{\leq 2}(X)}.
\end{equation*}
\end{thm}
\begin{proof}
This follows by applying Lemma~\ref{lem: global sections of loc sys is equivariantization} to the local system of categories $\cL$-$\mathfrak{mod}$, described in Lemma~\ref{lem: L-modules is local system of categories}.
\end{proof}

One can note that the equivalence in Theorem~\ref{thm: L-modules is equivariantization} is actually an equivalence of tensor (and, more generally, Tannakian) categories.

\begin{rem}
Let us try to make an equivalence in Theorem~\ref{thm: L-modules is equivariantization} as explicit as possible, unraveling definitions from Section~\ref{subsubsec: higher categories}. First, take the universal cover $\nu: \wti X \rightarrow X$. Then $\nu^* \cL$ is a Lie algebroid, and there is an equivalence between $\cL$-modules on $X$ and $\pi_1(X)$-equivariant $\nu^* \cL$-modules on $\wti X$. In particular, $\Coh^\cL (X)$ is the $\pi_1(X)$-equivariantization of $\Coh^{\nu^* \cL}(\wti X)$ (note that this is just an equivariantization under the action of a group, as defined e.g. in \cite[Chapters~2.7,~4.15]{EGNO15}, which is conceptually easier than the 2-groupoid equivariantization).

Now, on $\wti X$, we have the restriction to point functor $\Coh^{\nu^* \cL}(\wti X) \rightarrow \h_x\mods$, which is fully faithful. The group $\pi_2(\wti X, x)$ acts by automorphisms of the identity endofunctor of $\h_x\mods$, and $\Coh^{\nu^* \cL}(\wti X)$ is identified with the full subcategory of $\h_x\mods$, consisting of objects $V$, for which the image of $\pi_2(\wti X, x)$ in $\Aut_{\h_x} V$ is trivial.

Let us also point out that due to Whitehead, the 2-groupoid $\pi_{\leq 2}(X)$ admits an explicit combinatorial model in terms of a cross module, see~\cite{Noo07} for an overview.
\end{rem}

\begin{rem} \label{rem: action of pi_2 on vect is trivial}
Let us verify that in the case $\cL = \cT_X$, Theorem~\ref{thm: L-modules is equivariantization} indeed reduces to the well-known fact $\Coh^{\cT_X}(X) \simeq \Rep \pi_1(X)$.

In this case $\h_x\mods \simeq \Vect$. Automorphisms of the identity endofunctor of $\Vect$ are in correspondence with nonzero scalars. So if the action of $\pi_2(X)$ on the identity functor $\id_{\Vect}$, appearing in Theorem~\ref{thm: L-modules is equivariantization} is nontrivial, it acts nontrivially on every object of $\Vect$. This would mean that $\Vect^{\pi_{\leq 2}(X)}$ is empty. However, $\Coh^{\cT_X}(X)$ is nonempty, since $\cO_X$ is an object in it. It follows that the action of $\pi_2(X)$ is trivial, and our equivalence reduces to $\Coh^{\cT_X}(X) \simeq \Vect^{\pi_1(X)}$, as required.
\end{rem}

\begin{rem}
We stated the result for holomorphic (complex-analytic) setting due to our later applications, but in fact the proof works just as well for the case of smooth real manifolds. Our result might be of independent interest in this context. 
\end{rem}

\begin{rem}
Theorem~\ref{thm: L-modules is equivariantization} can be thought of as an analogy with the description of bundles, equivariant with respect to a transitive action of a simply connected Lie group $H$. This category is equivalent to the category of representations of the stabilizer subgroup $H_x$ at a point $x \in X$. 
Note that $\pi_1(H) = \pi_2(H) = 0$ implies that $\pi_1(X) = \pi_0(H_x)$, $\pi_2(X) = \pi_1(H_x)$. 
Under this analogy, $\h_x$ appearing in Theorem~\ref{thm: L-modules is equivariantization}, should be thought of as the Lie algebra of $H_x$, and one needs to add corrections, involving $\pi_1(X) = \pi_0(H_x)$ and $\pi_2(X) = \pi_1(H_x)$, which mimic the fact that $H_x$ could be non connected and non simply connected.


Assume $\pi_2(X) = 0$. Then Theorem~\ref{thm: L-modules is equivariantization} reduces to the equivalence $\Coh^\cL(X) \simeq (\h_x\mods)^{\pi_1(X, x)}$.
Note that (in the smooth real setting) an obstruction to integrability of $\cL$ to a Lie groupoid is closely related to $\pi_2(X)$, see \cite{CF03}. It would be interesting to investigate how this ties into the picture described above.
\end{rem}

\subsection{Regular singularities and GAGA} \label{subsec: gaga}
For D-modules, the theory of regular singularities is a classical subject initiated by Deligne, see \cite{Del70}. This theory was generalized to Lie algebroids by K\"allstr\"om, see \cite{Kae98}.

We fix a Lie algebroid $\cL$ on a complex algebraic variety $X$. Let $X^\an$ be the corresponding complex analytic space, and let $\cL^\an$ be the corresponding holomorphic Lie algebroid. We have a natural analytification functor $\Coh^\cL X \rightarrow \Coh^{\cL^\an} X^\an$, $\cF \mapsto \cF^\an$. The following result will be sufficient for the purposes of the present paper. 
\begin{prop} \label{prop: gaga for open of codim 2}
Suppose $\ov X$ is a normal projective complex algebraic variety, and $X \subset \ol X$ is open with complementary of codimension at least 2. Suppose ${(G, \cL)}$ is a HC Lie algebroid on $\ol X$, and $X$ is contained in its open orbit. Then the analytification functor
\begin{equation*}
\Coh^{(G, \cL)} X \rightarrow \Coh^{(G^\an, \cL^\an)} X^\an
\end{equation*}
is fully faithful.

Moreover, if $M \in \Coh^{(G, \cL)} X$, and $N^\an \subset M^\an$ is its holomorphic submodule, then there exists (algebraic) $N \in \Coh^{(G, \cL)} X$ whose analytification is isomorphic to $N^\an$.
\end{prop}
Vaguely speaking, this proposition says that any coherent $\cL$-module on $X$ in the above situation has regular singularities.

\begin{proof}
Denote by $j: X \hookrightarrow \ol X$ the open embedding. Take $M, N \in \Coh^{(G, \cL)} X$. By adjunction, we have 
\[
\Hom_{(G, \cL)}(j_* M, j_* N) \simeq \Hom_{(G, \cL)}(j^* j_* M, N) \simeq \Hom_{(G, \cL)}(M, N),
\]
where we use Lemma~\ref{lem: pushforward for open embedding}, which states that the algebroid direct and inverse images coincide with the usual quasi-coherent versions for an open embedding. 
Similarly in the analytic category, we have 
\[
\Hom_{(G^\an, \cL^\an)}((j^\an)_* M^\an, (j^\an)_* N^\an) \simeq \Hom_{(G^\an, \cL^\an)}(M^\an, N^\an).
\]

By Lemma~\ref{coherent module is locally free}, $M$ is locally free, hence torsion-free. Using the codimension $\geq 2$ assumption,  \cite[Lemma 5.1.1 (2)]{Kae98} implies that in the algebraic category, $j_* M$ is $\cO_{\ol X}$-coherent (this is a variant of the Grothendieck finiteness theorem \cite[VIII, Corollaire~2.3]{Gro68}).
By a result of Serre \cite[Section~6, Remarque~2)]{Ser66}, it follows that $(j_* M)^\an \simeq (j^\an)_*(M^\an)$. Similarly for~$N$.

On projective $\ol X$, a variant of GAGA \cite{Ser56} for Lie algebroids \cite[Theorem~4.1.1.]{Kae98}, in particular, implies that the analytification is fully faithful, hence $\Hom_{(G, \cL)}(j_* M, j_*N) \simeq \Hom_{(G^\an, \cL^\an)}((j_* M)^\an, (j_*N)^\an)$ (strictly speaking, \cite{Kae98} deals only with the case $G = \id$, but the generalization to the equivariant setting is straightforward).

Combining all of the above, we have:
\begin{multline*}
\Hom_{(G, \cL)}(M, N) \simeq \Hom_{(G, \cL)}(j_* M, j_*N) \simeq \Hom_{(G^\an, \cL^\an)}((j_* M)^\an, (j_*N)^\an) \simeq \\
\Hom_{(G^\an, \cL^\an)}((j^\an)_* M^\an, (j^\an)_* N^\an) \simeq \Hom_{(G^\an, \cL^\an)}(M^\an, N^\an),
\end{multline*}
and fully faithfulness is proved.

Let us prove the second claim. From the inclusion $N^\an \subset M^\an$ and the left-exactness of direct image, we have $(j^\an)_* N^\an \subset (j^\an)_* M^\an \simeq (j_* M)^\an$; the last isomorphism, as well as $\cO$-coherence of $(j_* M)^\an$ is justified above. It follows that $(j^\an)_* N^\an$ is $\cO$-coherent, and therefore, due to \cite[Theorem~4.1.1]{Kae98}, lies in the essential image of the analytification functor on $\ol X$. It follows that its restriction to $X$ is also algebraic, as required.
\end{proof}

\begin{cor} \label{cor: analytic galois to algebraic is faithfully flat gaga}
In the setup of Proposition~\ref{prop: gaga for open of codim 2}, the homomorphism of pro-algebraic groups $\Gal^{\cL^\an}(X^\an) \rightarrow \Gal^\cL(X)$, Tannakian-dual to the analytification functor on $X$, is faithfully flat.
\end{cor}
\begin{proof}
By \cite[Proposition~2.21(a)]{DM82}, this is equivalent to Proposition~\ref{prop: gaga for open of codim 2}.
\end{proof}

\subsection{Derived category of equivariant sheaves} \label{subsection: Derived category of equivariant sheaves}

An important difference between the category of sheaves equivariant under the action of a (Harish-Chandra) Lie algebroid and the category of sheaves equivariant under the action of a Lie group (considered, in particular, in \cite{Bez00}), is that in the case of Lie algebroid the inclusion of derived categories $D^b \Coh^{(G, \cL)} (X) \subset D^b_{\mathrm{coh}} \qcoh^{(G, \cL)} (X)$ (the latter being the category of complexes of quasi-coherent sheaves with $\cO$-coherent cohomology) does not induce an equivalence. Even in the case when $X$ is a point, $G = \id$, and $\cL$ is a semi-simple Lie algebra, that is known not to be the case:
\begin{example} \label{example: g-modules on point}
For a simple Lie algebra $\g$, the category of finite-dimensional modules is semi-simple, and all higher $\Ext$'s vanish. At the same time, in the category of all $\g$-modules, higher $\Ext$'s appear as Lie algebra cohomology groups, and can be nontrivial even for finite-dimensional modules.
This demonstrates that $D^b \Coh^{\g} (\pt) \neq D^b_{\mathrm{coh}} \qcoh^{\g} (\pt)$.
\end{example}

Since the aim of this paper is the construction of perverse t-structure, and one of the applications of it is in obtaining the basis of simple objects in $K$-theory, we now show that at the level of Grothendieck groups these two categories are the same:

\begin{lem} \label{K-theory of coh vs qcoh}
    There is a canonical isomorphism $K_0 (\Coh^{(G, \cL)} (X)) \simeq K_0(D^b_{\mathrm{coh}} \qcoh^{(G, \cL)} (X))$.
\end{lem}
\begin{proof}
One easily sees that the morphisms
\begin{align*}
K_0 (\Coh^{(G, \cL)} (X)) &\longleftrightarrow K_0(D^b_{\mathrm{coh}} \qcoh^{(G, \cL)} (X)) \\
[\cF] &\longmapsto [\cF[0]] \\
\sum_{i \in \bZ} (-1)^i [ H^i(\cG)] &\longmapsfrom [\cG]
\end{align*}
are well-defined and inverse to each other.
\end{proof}

We denote the $K$-group from Lemma~\ref{K-theory of coh vs qcoh} by $K^{(G, \cL)}(X)$.

For technical reasons, we work with the category $D^b_{\mathrm{coh}} \qcoh^{(G, \cL)} (X)$ rather than $D^b \Coh^{(G, \cL)} (X)$. 
We denote $D^{b,  (G, \cL)}_{\mathrm{coh} } (X) := D^b_{\mathrm{coh}} \qcoh^{(G, \cL)} (X)$, $D^{b}_{\mathrm{coh} } (X) := D^b_{\mathrm{coh}} \qcoh (X)$ and similarly for other categories.

As explained in \cite[Section~3]{Kae98}, there are enough injectives in $\qcoh^{(G, \cL)}(X)$, and the functors of inverse and direct image are defined at the level of derived categories (formally speaking, only the non-equivariant case is considered in \emph {loc. cit.}, but adding group-equivariance is standard --- as in the case of D-modules).

Let $i_Z: Z \rightarrow X$ be a locally closed embedding of a $\cL$-invariant subscheme. We have natural functors of direct and inverse images, described in Lemma~\ref{direct and inverse images for locally closed embedding}.
Clearly, they preserve coherence if and only if their derived functors preserve the coherence of cohomology. 

From now on, we use these notation for the derived functors
\[
i_Z^*, i_Z^!: D^{b, {(G, \cL)}}_{\mathrm{coh}} (X) \rightarrow D^{b,{(G, \cL)}}_{\mathrm{coh}} (Z),
\]
and, if $i_Z$ is a closed embedding,
\[
i_{Z*}: D^{b, {(G, \cL)}}_{\mathrm{coh}} (Z) \rightarrow D^{b, {(G, \cL)}}_{\mathrm{coh}} (X),
\]
see the discussion in Section~\ref{subsec: direct and inverse image}.



\subsection{Equivariant dualizing object} \label{subsection equivariant dualizing object}

We call an object $\om_X \in D^{b, (G, \cL)}_{\mathrm{coh}}(X)$ an {\bf equivariant dualizing object} if for any $\cF \in D^{b, (G, \cL)}_{\mathrm{coh}}(X)$ the natural homomorphism \[\cF \rightarrow R \shhom(R \shhom(\cF, \om_X), \om_X)\] is an isomorphism.

Let $\Forg: D^{b, (G, \cL)}_{\mathrm{coh}}(X) \rightarrow D^b_{\mathrm{coh}}(X)$ be the functor that forgets the $(G, \cL)$-action. 
\begin{lem} \label{dualizing object iff Forg is dualizing}
$\om_X \in D^{b, (G, \cL)}_{\mathrm{coh}}(X)$ is an equivariant dualizing object if and only if $\Forg \om_X \in D^{b}_{\mathrm{coh}}(X)$ is a dualizing object.
\end{lem}
\begin{proof}
In line with \cite[Lemma 4]{Bez00}, which is the same result for the group-equivariant category.
\end{proof}

We expect $\cO$-coherent dualizing object, equivariant under a HC Lie algebroid, to exist in a large class of situations. 
However, the schemes of interest in the present paper are Gorenstein (see Section~\ref{Section: symplectic singularities}), hence the existence of an equivariant dualizing object will be evident from Lemma~\ref{dualizing object iff Forg is dualizing}. That is why we do not pursue the general existence question further, and throughout Section \ref{sec: perverse coherent modules over HC Lie algebroids}, we simply assume that there exists an equivariant dualizing object $\om_X \in D^{b, (G, \cL)}_{\mathrm{coh}}(X)$. 
We denote the duality functor by $\bD(-) = R\shhom(-, \om_X)$.

\section{Perverse coherent modules over Harish-Chandra Lie algebroids} \label{sec: perverse coherent modules over HC Lie algebroids}

Recall the construction of the perverse coherent t-structure of \cite{Bez00, AB10} (note that similar constructions appear in \cite{Kas03}, \cite{Gab04}).
The goal of this section is to repeat it for the case of modules over a HC lie algebroid.
The construction is almost verbatim the same as in \cite{Bez00, AB10}, and we need to transfer all the lemmata from these papers to our setting. Often, the proofs do not depend on the equivariance condition. 
However, they are still formally new in this setting, so we include their full statements. For the proofs, we refer to \cite{Bez00, AB10} whenever possible, but provide full proofs in cases where the arguments differ in our setting. Note that the original setting of \cite{Bez00} of group-equivariant sheaves is closer to ours, but the proofs in \cite{AB10} (in the setting of algebraic stacks) are very similar (and usually are better-written, which is why we prefer to refer to \cite{AB10}).

\subsection{Preparatory lemmata}

In this subsection, we repeat all the facts from \cite{Bez00, AB10} needed for the construction of the perverse t-structure, in our setting of sheaves equivariant under a HC Lie algebroid.

Let $(G, \cL)$ be a Harish-Chandra Lie algebroid on $X$. Recall that $D^{b, (G, \cL)}_\coh(X)$ is the triangulated category of $(G, \cL)$-modules with $\cO$-coherent cohomology.



Consider the topological space $(X / \cL)^{\textrm{top}}$, whose points are the generic points of $\cL$-invariant closed reduced subschemes of $X$, equipped with the natural (Zariski) topology. For $x \in (X / \cL)^{\textrm{top}}$ we denote by $\dim x$ the dimension of the corresponding subscheme of $X$.

\begin{lem}[\cite{AB10}, Lemma 2.13] \label{sheaf with support on closed is pushforward from subscheme}
Given a closed $\bold Z \subset (X / \cL)^{\mathrm{top}}$ and $\cF \in D_{\mathrm{coh}}^{b, (G, \cL)} (X)$ such that $\supp \cF \subset \bold Z$, there exists a closed $(G, \cL)$-invariant subscheme $i_Z: Z \hookrightarrow X$ with $Z^{\mathrm{top}} \subset \bold Z$ and $\cF_Z \in D_{\mathrm{coh}}^{b, (G, \cL)} (Z)$ such that $\cF = i_{Z *} \cF_Z$.
\end{lem}

Note that the proof in \cite{AB10} uses the realization of an object in this triangulated as a complex of coherent sheaves, which cannot be done w.r.t. Lie algebroid action, as noted above. We provide an argument that reduces the algebroid-equivariant case to the non-equivariant case.

\begin{proof}

The proof proceeds by induction on the number of nonzero cohomology groups of $\cF$.

If this number equals 1, then our claim becomes the corresponding claim for the abelian category of coherent $(G, \cL)$-modules, which is clear.

For the induction step, we note that there is a distinguished triangle in $D_{\mathrm{coh}}^{b, (G, \cL)} (X)$
\begin{equation*}
\cF_2 \rightarrow \cF \rightarrow \cF_1 \rightarrow \cF_2[1]
\end{equation*}
such that $\cF_1, \cF_2$ have fewer nonzero cohomology groups than $\cF$, and all these cohomology sheaves are supported on $\bold Z$. Hence by the induction assumption, $\cF_1, \cF_2$ can be realized as direct images from some closed $(G, \cL)$-invariant subscheme $Z$. $\cF$ represents some class in the $\Ext$-group between these sheaves. Thus, it is sufficient to prove that $\Rhom$ between sheaves supported on $\bold Z$, in the category of $(G, \cL)$-modules on $X$, coincides with limit of $\Rhom$'s in the category of $(G, \cL)$-modules on thickenings of $Z$.

More formally, denote by $R\Hom_{(G, \cL)}$ morphisms in the category $D_{\mathrm{coh}}^{b, (G, \cL)} (X)$. Then we have
\begin{equation} \label{eq:hom is composition of inner hom, invariants and sections}
R \Hom_{(G, \cL)}(\cF_1, \cF_2) \simeq R\Gamma \circ R(-)^{(G, \cL)} \circ R \shhom_{\cO_X} (\cF_1, \cF_2),
\end{equation}
where $R(-)^{(G, \cL)} = R \shhom_{(G, \cL)}(\underline \bC, -)$ is the derived functor of $(G, \cL)$-invariants. 

Denote by $Z_n$ the $n$-th formal neighborhood of $Z$ in $X$, and $i_n: Z \hookrightarrow Z_n$ the inclusion. Abusing notation, we view $\cF_i$ $(i = 1, 2)$ as sheaves on $X$ or on $Z$, depending on the context. The statement of the Lemma in the non-equivariant case (\cite[Lemma 2.3]{AB10}) in particular tells us that
\begin{equation} \label{eq:hom is colimit over thickenings}
R \shhom_{\cO_X} (\cF_1, \cF_2) = \varinjlim_n R\shhom_{\cO_{Z_n}}(i_{n *}  \cF_1, i_{n *} \cF_2).
\end{equation}
Note that in our case $R\Gamma$ commutes with filtered colimit \cite[\href{https://stacks.math.columbia.edu/tag/0738}{Tag 0738}]{SP};
$R(-)^{(G, \cL)} = R \shhom_{(G, \cL)}(\underline \bC, -)$ also commutes with filtered colimits. Hence, \eqref{eq:hom is composition of inner hom, invariants and sections} and \eqref{eq:hom is colimit over thickenings} yield
\begin{equation} \label{eq:distinguished triangle for F}
R\Hom_{X, (G, \cL)}(\cF_1, \cF_2) \simeq \varinjlim_n R\Hom_{Z_n, (G, \cL)} (i_{n *} \cF_1, i_{n *} \cF_2).
\end{equation}

This implies that the class of $\cF$ in \eqref{eq:distinguished triangle for F} comes from some $Z_n$, as required.
\end{proof}

For $x \in (X / \cL)^{\textrm{top}}$, denote $\bold i_x: \{ x \} \rightarrow (X / \cL)^{\textrm{top}}$. We have derived functors:
\begin{align*}
\bold i_x^*: D^{b, {(G, \cL)}}_{\mathrm{coh} }  X &\rightarrow D^b_{\mathrm{coh}} (\cO_x \text {--mod}), \\
\bold i_x^!: D^{b, {(G, \cL)}}_{\mathrm{coh}}  X &\rightarrow D^b_{\mathrm{coh}} (\cO_x \text {--mod}).
\end{align*}
Note that $\bold i_x^!$ has finite cohomological dimension by \cite[Lemma 2.19]{AB10}.

The following lemmata of \cite{Bez00, AB10} are independent of the equivariance condition, hence applicable to our setting.

\begin{lem}[\cite{AB10}, Lemma 2.21]
Let $Z\subset X$ be a locally closed subscheme, and $n$ be an integer. Let $x \in X^\tp$ be a generic point of $Z$. Then:
\begin{enumerate}[(a)]
\item\label{equiv:a}  For $\cF \in D^b_{\mathrm{coh}}(X)$ we have $\bi_x ^*(\cF)\in D^{\leq n}(\cO _x\md)$ if and only if
there  exists an open subscheme $Z_0\subset Z$, $Z_0 \owns x$
 such that $i_{Z_0}^*(\cF)\in D^{\leq n}_{\mathrm{coh}} (Z_0) $;

\item\label{equiv:b}  For $\cF \in D^b_{coh}(X)$ we have $\bi_x ^!(\cF)\in D^{\geq n}(\cO_x\md)$ if and only if there exists an open  subscheme $Z_0\subset Z$, $Z_0\owns x$
such that $i_{Z_0}^!(\cF )\in D^{\geq n}_{coh} \qcoh Z_0$.
\end{enumerate}
\end{lem}

\begin{lem}[\cite{AB10}, Lemma 2.22]
Let $\bi: \bold Z \hookrightarrow X^\tp$ be the embedding of a closed
subspace. 
For any $\cF \in D^-_{\mathrm{coh}}(X)$, $\cG \in D^+_{\textrm{qcoh}}(X)$ we have
$$
\Hom(\cF ,\bi_*\bi^!(\cG))=\varinjlim_Z \Hom(\cF ,i_{Z*}i_Z^!(\cG)),
$$                                 
where $Z$ runs over the set of closed subschemes of $X$ with underlying
topological space $\bold Z$.

\end{lem}

\subsection{Perverse coherent t-structure}
The lemmata in the previous subsection are the only ingredients needed in the proofs in \cite[Section~3.1]{AB10} of the fact that the perverse coherent t-structure is indeed a t-structure. We adapt the arguments for our setting in this subsection.

From now on, we assume that there exists an equivariant dualizing object in $D^{b, (G, \cL)}_\coh(X)$.

Let $p: (X / \cL)^\tp \rightarrow \bZ$ be a perversity (an integer-valued function on $(X / \cL)^\tp$). We define the dual perversity as $\ol p(x) = -\dim x - p(x)$.

We define the full subcategories $D^{p, \leq 0} \subset D_{\mathrm{coh}}^{b, (G, \cL)}(X)$, $D^{p, \geq 0} \subset D_{\mathrm{coh}}^{b, (G, \cL)}(X)$ by:
\begin{equation} \label{definition of D^+ and D^-}
\begin{aligned}
     \cF \in D^{p, \geq 0} &\text{ if for any }x \in (X / \cL)^{\textrm{top}} \text{ we have } \bold i_x^! \cF \in D^{\geq p(x)}(\cO_x \text {--mod}). \\
     \cF \in D^{p, \leq 0} &\text{ if for any } x \in (X / \cL)^{\textrm{top}} \text{ we have } \bold i_x^* \cF \in D^{\leq p(x)}(\cO_x \text {--mod}).
\end{aligned}
\end{equation}

Now, the proof that this defines the t-structure is completely analogous to \cite{AB10}. We state all the steps below.

\begin{lem} \label{basic properties of perversity}
One has:
\begin{enumerate} [(a)]
\item\label{basic:duality} $\bD(D^{p, \leq 0}(X))=D^{\ol p,\geq 0}(X)$.

\item\label{} Let  $i_Z:Z \hookrightarrow X$ be a locally closed 
subscheme. Define the \emph{induced perversity} on $Z$ by $p_Z=p\circ i_Z: Z^\tp\to \bZ$.
Then 
$$i_Z^*(D^{p, \leq 0}(X)) \subset D^{p_Z,\leq 0}(Z)\qquad\text{and}\qquad i_Z^!(D^{p, \geq 0}(X)) \subset D^{p_Z, \geq 0}(Z).$$

\item\label{predvar:c} In the situation of (b), assume that $Z$ is closed. Then 
$$i_{Z*}(D^{p_Z,\leq 0}(Z))\subset D^{p, \leq 0}(X) \qquad\text{and} \qquad i_{Z*}(D^{p_Z,\geq 0}(Z)) \subset D^{p, \geq 0}(X).$$ 
\end{enumerate}
\end{lem}
\begin{proof}
In line with \cite[Lemma 3.3]{AB10}.
\end{proof}

\begin{lem} \label{Homs are zero between p+ and p-}
For $\cF \in D^{p, \leq 0}(X)$, $\cG \in D^{p, \geq 0}(X)$ we have $\Hom(\cF, \cG) = 0$.
\end{lem}
\begin{proof}
In line with \cite[Proposition 3.5]{AB10}
\end{proof}

\begin{definition} \label{definition monotone and comonotone}
We say that the perversity $p$ is \textbf{monotone} if $p(x') \geq p(x)$ whenever $x' \in \ol{\{x\}}$. We say it is \textbf{strictly monotone} if $p(x') > p(x)$ whenever $x' \in \ol{\{x\}}$ and $x' \neq x$. Finally, we say it is \textbf{(strictly) comonotone} if the dual perversity $\ol p(x) = - \dim x - p(x)$ is (strictly) monotone.
\end{definition}

\begin{lem} \label{there is distinguished triangle with perverse truncation}
Suppose $p$ is monotone and comonotone. Then for any $\cF \in D^{b, (G, \cL)}_{\mathrm{coh}}(X)$ there are $\cF_1 \in D^{p, \leq 0} (X)$, $\cF_2 \in D^{p, > 0} (X)$ and a distinguished triangle
\[
\cF_1 \rightarrow \cF \rightarrow \cF_2 \rightarrow \cF_1[1]
\]
\end{lem}
\begin{proof}
In line with \cite[Theorem 1]{Bez00}, replacing ``$G$-equivariant'' by ``$(G, \cL)$-equivariant'' in appropriate places (see also \cite[Theorem 3.10]{AB10}).
\end{proof}

Combining Lemmata~\ref{basic properties of perversity},~\ref{Homs are zero between p+ and p-},~\ref{there is distinguished triangle with perverse truncation}, we obtain:
\begin{thm} \label{conditions for perverse t-structure define t-structure}
    The formulae \eqref{definition of D^+ and D^-} define a t-structure on the category $D^{b, (G, \cL)}_{\mathrm{coh}} (X)$. Its heart is 
    denoted by $\cP^{(G, \cL)}_{\coh} (X)$ and is called the category of \textit{perverse coherent sheaves}. 

\end{thm}

\subsection{IC-extension} \label{subsection: IC-extension}
In this subsection, we define the notion of coherent IC-extension, as in  \cite[Section 3.2]{Bez00}, \cite[Section 4.1]{AB10}. 

From now on, we assume that the perversity $p$ is strictly monotone and strictly comonotone (see Definition \ref{definition monotone and comonotone}).

Let $Y \hookrightarrow X$ be a locally closed $(G, \cL)$-invariant subscheme, which we decompose as the composition of open and closed embeddings $Y \hookrightarrow \ol Y \hookrightarrow X$ (note that $\ol Y$ is not assumed to be reduced here).
Denote by $\bold Z \subset (X / \cL)^\tp$ the topological space of $(\ol Y \setminus Y)$.

Define $p^+ = p^+_{\bold Z}, p^- = p_{\bold Z}^-: (\ol Y / \cL)^{\text{top}} \rightarrow \bZ$ as: 
$$
p^-(x)=\begin{cases}p(x), &x \not \in \bold Z \\
p(x) - 1, & x \in \bold Z; \end{cases} \qquad
p^+(x) = \begin{cases} p(x), & x \not\in \bold Z \\ 
p(x) + 1, & x \in \bold Z. \end{cases}
$$
Since $p$ is strictly monotone and strictly comonotone, both $p^+$ and $p^-$ are monotone and comonotone.

\begin{lem} \label{equivalent conditions for being in p^+ and p^-}
 Let $\cF \in \cP^{(G, \cL)}_\coh (\ol Y)$.
\begin{enumerate}[(a)]
\item \label{orto:a} The following conditions are equivalent:
\begin{enumerate}[(i)]
\item\label{orto:i} $\cF \in D^{p^-,\leq 0}(\ol Y)$.
\item\label{orto:ii} $i_Z^*(\cF)\in D^{p_{Z},< 0}(Z)$ for any closed subscheme $Z \subset \ol S$ with $ Z^\tp \subset \bold Z$.
\item\label{orto:iii} $\Hom(\cF, \cG)=0$ for all $\cG \in \cP^{(G, \cL)}_\coh (\ol Y)$ such that $\supp \cG \subset \bold Z$.
\end{enumerate}
\item\label{orto:b} The following conditions are equivalent:
\begin{enumerate}[(i)]
\item $\cF\in D^{p^+,\geq 0}(\ol Y)$.
\item $i_Z^!(\cF) \in D^{p_{Z}, >0}(Z)$ for any closed subscheme $Z \subset \ol Y$ with $Z^\tp \subset \bold Z$.
\item $\Hom(\cG, \cF) = 0$ for all $\cG \in \cP^{(G, \cL)}_\coh (\ol Y)$ such that $\supp \cG \subset \bold Z$.
\end{enumerate}
\end{enumerate}
\end{lem}

\begin{proof}
In line with \cite[Lemma 4.1]{AB10}.
\end{proof}

From now on, we assume that there is a finite number of $(G, \cL)$-orbits in $X$, and dimensions of adjacent orbits differ at least by 2, that is, for any $(G, \cL)$-orbit $S$, one has
\begin{equation} \label{codim 2}
\codim_{\ol S} S \geq 2.
\end{equation}
We restrict our attention to the IC-extension from a single $(G, \cL)$-orbit $S$ to its closure $\ol S$.
Denote by $\bold Z \subset (X / \cL)^\tp$ the topological space of $\ol S \setminus S$.

\begin{lem} \label{serre quotient by subcategory of sheaves supported on closed subset}
Under assumption~\eqref{codim 2}, the category $\Coh^{(G, \cL)} S$ is equivalent to the Serre quotient of $\Coh^{(G, \cL)} \ol S$ by the subcategory $\Coh^{(G, \cL)}_{\bold Z} \ol S$ of sheaves, supported on $\bold Z$:
\[
\Coh^{(G, \cL)} S \simeq \Coh^{(G, \cL)} \ol S \left/ \Coh^{(G, \cL)}_{\bold Z} \ol S \right. .
\]
\end{lem}
Note that this is a standard fact (even without assumption \eqref{codim 2}) for non-equivariant sheaves, or group-equivariant sheaves. The standard proof uses that a quasi-coherent sheaf is a union of coherent subsheaves, which is not the case for equivariance over Lie algebroids, as noted above. Thus, we include a proof.

\begin{proof}
Let $j: S \subset \ol S$ be the open embedding. $j^*$ naturally induces the faithful exact functor
\begin{equation} \label{quotient functor by subcategory of sheaves, supported on closed}
\Coh^{(G, \cL)} \ol S \left/ \Coh^{(G, \cL)}_{\bold Z} \ol S \rightarrow  \Coh^{(G, \cL)} S \right. .
\end{equation}

Let $\cF \xrightarrow{f} \cG$ be a morphism in $\Coh^{(G, \cL)} S$. By Lemma~\ref{coherent module is locally free}, $\cF$ and $\cG$ are locally free, and hence torsion-free. Using \eqref{codim 2},  by \cite[Lemma 5.1.1 (2)]{Kae98}, this implies that non-derived direct images $R^0j_*(\cF)$ and $R^0j_* (\cG)$ are coherent (again, this is a variant of 
\cite[VIII, Corollaire 2.3]{Gro68}).

Thus, $\cF \xrightarrow{f} \cG = j^*(R^0j_* \cF \xrightarrow{R^0j_* f} R^0j_* \cG)$, hence \eqref{quotient functor by subcategory of sheaves, supported on closed} is full and essentially surjective, which finishes the proof.
\end{proof}


Define the category $\cP_{!*}(S) \subset \cP^{(G, \cL)}_{\coh} (\overline S)$ as \[\cP_{!*}(S) = D^{p^-, \leq 0}(\ol S) \cap D^{p^+, \geq 0}(\ol S).\]

Let $s$ be the generic point of $S$. Then the category $\cP_{\coh}^{(G, \cL)}(S)$ is naturally equivalent to $\Coh^{(G, \cL)}(S)[-p(s)]$ --- the abelian category of coherent sheaves on $S$, put in the cohomological degree $[-p(s)]$.
Consider also the abelian category $\Coh^{(G, \cL)}(\ol S)[-p(s)]$.

By results of previous subsection, $p^+$ and $p^-$ define t-structures on $D^{b, (G, \cL)}_\coh (\ol S)$. Let $\tau^+, \tau^-$ be the corresponding truncation functors. Define the functor 
\begin{align*}
J_{!*}: \Coh^{(G, \cL)}(\ol S)[-p(s)] &\rightarrow D^{b, (G, \cL)}_\coh (\ol S) \\
\cF &\mapsto \tau_{\leq 0}^- \circ \tau^+_{\geq 0} \cF.
\end{align*}
\begin{lem}[\cite{AB10}, Lemma 4.3] \label{functor J_!*}
The following holds:
\begin{enumerate}[(a)]
    \item $J_{!*}$ takes values in $\cP_{!*}(S)$.
    \item If a morphism $f$ in the category $\Coh^{(G, \cL)}(\ol S)[-p(s)]$ is such that $f|_S$ is an isomorphism, then $J_{!*}f$ is an isomorphism
\end{enumerate}
\end{lem}
\begin{proof}
In line with \cite[Lemma 4.3]{AB10} (in fact, it holds for $J_{!*}$ defined on the whole category $D^{b, (G, \cL)}_\coh (\ol S)$, hence for its restriction to $\Coh^{(G, \cL)}(\ol S)[-p(s)]$).
\end{proof}

\begin{thm}
Functor $j^*$ induces an equivalence between $\cP_{!*}(S)$ and $\cP^{(G, \cL)}_\coh(S)$. 

The inverse equivalence
\[
\cP^{(G, \cL)}_\coh(S) \rightarrow \cP_{!*}(S) \subset \cP^{(G, \cL)}_\coh(\ol S)
\]
is denoted by $j_{!*}$ or $\IC(S, -)$, and is called the coherent Goresky--MacPherson or $\IC$ extension.
\end{thm}

\begin{proof}
Lemma \ref{serre quotient by subcategory of sheaves supported on closed subset} tells that the functor 
\[
\Coh^{(G, \cL)}(\ol S)[-p(s)] \xrightarrow{j^*} \Coh^{(G, \cL)}(S)[-p(s)] = \cP^{(G, \cL)}_\coh(S)
\]
is the Serre quotient by the subcategory $\Coh^{(G, \cL)}_{\bold Z}(\ol S)[-p(s)]$. Lemma~\ref{functor J_!*} tells that the functor $J_{!*}$ factors through this quotient, i.e., it canonically decomposes as
\[
\Coh^{(G, \cL)}(\ol S)[-p(s)] \rightarrow \cP^{(G, \cL)}_\coh(S) \rightarrow \cP_{!*}(S).
\]
We denote this last functor by $j_{!*}: \cP^{(G, \cL)}_\coh(S) \rightarrow \cP_{!*}(S)$. By construction, $j^* \circ j_{!*} = \id_{S}$ canonically. We also have $j_{!*} \circ j^* = \id_{S}$ canonically, since $J_{!*}\vert_{\cP_{!*}(S)} = \id$. This finishes the proof.
\end{proof}

We also immediately obtain
\begin{lem}[\cite{AB10}, Lemma 4.4] \label{IC of restriction is subquotient}
For any $\cF \in \cP^{(G, \cL)}_{\coh}(\ol S)$, $j_{!*}(\cF\vert_{S})$ is a subquotient of $\cF$ in the abelian category $\cP^{(G, \cL)}_{\coh}(\ol S)$.
\end{lem}
\begin{proof}
In line with \cite[Lemma 4.4]{AB10}.
\end{proof}

Finally, if $h: S \rightarrow X$ is the locally closed embedding of an orbit, which decomposes as $S \xrightarrow{j} \ol S \xrightarrow{i} X$, we define
\[
\IC(S, -) = h_{!*} = i_* \circ j_{!*}.
\]

\subsection{Irreducible perverse coherent sheaves}
We keep assuming that the perversity $p$ is strictly monotone and strictly comonotone, as well as the condition \eqref{codim 2}. We prove the following

\begin{prop} \label{classification of simple perverse coherent sheaves}
$\cF$ is an irreducible object in $\cP^{(G, \cL)}_\coh(X)$ if and only if it has the form $\IC(S, V[-p(S)])$, where $S$ is a $(G, \cL)$-orbit, and $V$ is an irreducible coherent $(G, \cL)$-module on $S$.
\end{prop}

Note that due to Section~\ref{subsection differential galois group}, the category of coherent $(G, \cL)$-modules on $S$ is Tannakian. In particular, the set of its irreducible objects is the set of irreducible representations of some pro-algebraic group.

\begin{proof}
We first show that $\IC(S, V[-p(S)])$ is irreducible in $\cP^{(G, \cL)}_\coh(X)$. Suppose $\cF' \subset \IC(S, V[-p(S)])$. Then $\supp \cF' \subset (\ol S)^\tp$. By Lemma \ref{sheaf with support on closed is pushforward from subscheme}, there is a $(G, \cL)$-invariant subscheme $Y$, such that $Y^\tp \subset (\ol S)^\tp$ and $\cF$ is the direct image of some sheaf on $Y$. 
If $S \subset Y$, we have $\cF' \vert_S \subset V[-p(S)]$, which together with Lemma~\ref{IC of restriction is subquotient} implies that $\cF' = \IC(S, V[-p(S)])$, since $V$ is irreducible (recall that $\cP^{(G, \cL)}_\coh(S) = \Coh^{(G, \cL)}(S)[-p(S)]$). If $Y^\tp \subset (\ol S \setminus S)^\tp$, then Lemma~\ref{equivalent conditions for being in p^+ and p^-} implies the result.

Now we show the converse direction. Assume $\cF$ is some irreducible object of $\cP^{(G, \cL)}_\coh(X)$. Let $x$ be a generic point of $\supp \cF$. Let $i_Z: Z \hookrightarrow X$ be a closed $(G, \cL)$-invariant subscheme of $X$, not containing $x$. Since $\cF$ is irreducible, we have
\[
\Hom(\cF, i_{Z *} \cG) = 0, \qquad \Hom(i_{Z *} \cG, \cF) = 0
\]
for any $\cG \in \cP^{(G, \cL)}_\coh (Z)$. Lemma \ref{equivalent conditions for being in p^+ and p^-} thus implies that
\[
i_Z^* \cF \in D^{p_Z, < 0}(Z), \qquad i_Z^! \cF \in D^{p_Z, > 0}(Z).
\]
It follows that $\supp \cF$ is irreducible (otherwise we could take $Z$ to lie inside an irreducible component of $\supp \cF$ not containing $x$, and obtain a contradiction, as $i_Z^* \cF = i_Z^! \cF$ in this case).

By Lemma~\ref{sheaf with support on closed is pushforward from subscheme}, $\cF$ can be obtained as a direct image from a closed $(G, \cL)$-invariant subscheme $Y \hookrightarrow X$ with generic point $x$. By finiteness of orbits, there is an open $(G, \cL)$-invariant $j : S \subset Y$. Lemma \ref{IC of restriction is subquotient} then implies that $\cF = \IC(S, V[-p(S)])$, where $V[-p(S)] = j^* \cF$. It is clear that $V$ is irreducible in this case, as required.
\end{proof}

\begin{cor} \label{perverse coherent sheaves are artinian and noetherian}
The category of perverse coherent sheaves is Artinian and Noetherian. Classes of irreducible perverse coherent sheaves form a basis of the K-group $K^{(G, \cL)}(X)$.
\end{cor}
\begin{proof}
Same as in \cite[Corollary 4.13, Corollary 4.14]{AB10}, using induction on the number of orbits.
\end{proof}

To sum up, here is the main theorem of this section, combined from Theorem \ref{conditions for perverse t-structure define t-structure}, Proposition \ref{classification of simple perverse coherent sheaves} and Corollary \ref{perverse coherent sheaves are artinian and noetherian}.

\begin{thm} \label{main theorem on perverse coherent sheaves}
Suppose $X$ is a finite type scheme and $(G, \cL)$ is a Harish-Chandra Lie algebroid, acting on $X$ with finite number of orbits; dimensions of adjacent orbits differ at least by 2; there is a dualizing object in $D_\coh^{b, (G, \cL)} X$; $p$ is a strictly monotone and strictly comonotone perversity on $(X / \cL)^\tp$.

Then formulae \eqref{definition of D^+ and D^-} define the t-structure on the category $D_\coh^{b, (G, \cL)} X$. Its heart, denoted $\cP^{(G, \cL)}_{\coh} (X)$, and called the category of \textit{perverse coherent sheaves}, is an abelian Artinian Noetherian category. 

Isomorphism classes of simple objects in this category are in correspondence with pairs $(S, V)$, where $S$ is a $(G, \cL)$-orbit and $V$ is a $(G, \cL)$-equivariant vector bundle on $S$. Their classes form a basis of $K^{(G, \cL)}(X)$.
\end{thm}

\begin{example}
Suppose all orbits of $(G, \cL)$ have even dimension. Then $p(S) = \ol p (S) = - \frac{1}{2} \dim S$ is strictly monotone and strictly comonotone. It is called \textit{the middle perversity}. By Lemma~\ref{basic properties of perversity}\ref{basic:duality}, the category of perverse coherent $(G, \cL)$-equivariant sheaves is preserved by duality $\bD$ in this case.
\end{example}



\section{Perverse coherent sheaves for conical symplectic singularities} \label{Section: symplectic singularities}

\subsection{Symplectic singularities} \label{subsec: symplectic singularities}
Let $X$ be a conical symplectic singularity (not necessarily admitting a symplectic resolution). 
In particular, $X = \Spec A$ is an affine Poisson variety, with Poisson bracket of some degree ${\bold d} > 0$. This means that $A = \bigoplus_{i \geq 0} A_i$ is graded, $A_0 = \bC$, and $\{A_i, A_j\} \subset A_{i + j - \bold d}$.
We denote by $\torus_\hbar$ the contracting torus, and by $0 \in X$ the attracting point of the $\torus_\hbar$ conical action.
We denote by $G$ the group of Poisson $\bC^\times_\hbar$-equivariant automorphisms, and set $\g$ to be its Lie algebra.

As in \cite{Kam22}, we assume that for all $0 < i < \bold d$ holds $A_i = 0$. This guarantees that $A_{> 0}$ is a Poisson ideal, and hence $\{0\} \subset X$ is a symplectic leaf. 

\begin{rem}
This last assumption holds for the majority of examples. It does not hold for $S^n \bA^2$, but holds for $(S^n \bA^2)'$ --- the space of $n$ points on $\bA^2$ with sum 0 
(so one has $S^n \bA^2 = (S^n \bA^2)' \times \bA^2$). This illustrates what kind of assumption this is.
\end{rem}




Let us define the Lie algebra $\fl = (A_{\geq \bold d}, \{\cdot, \cdot\})$ --- the vector space $A_{\geq \bold d}$ with Lie bracket being the Poisson bracket (we introduce a separate notation to distinguish it from the commutative algebra $A$).  It is a non-negatively graded Lie algebra by letting $\fl_i = A_{i + \bold d}$ for $i \geq 0$. $\fl_0 = A_{\bold d}$ is a Lie subalgebra, and $\fl_{\geq j}$ is a Lie algebra ideal for any $j \geq 1$. 




\begin{lem} \label{A_d = g}
One has an isomorphism of Lie algebras $\fl_0 \simeq \g$, compatible with the actions of these Lie algebras on $X$.
\end{lem}
\begin{proof}
Let $m \geq d$ be such that $A$ is generated by $A_{\leq m}$ as algebra. Then $G$ is the algebraic subgroup of $\prod_{i = 1}^{2m} GL(A_i)$, which preserves the multiplication tensors (elements of $A_i^* \otimes A_j^* \otimes A_{i + j}$ for all $i, j \leq m$) and the Poisson bracket tensors (elements of $A_i^* \otimes A_j^* \otimes A_{i + j - d}$ for all $i, j \leq m$). 
Then $\Lie G = \g$ is the Lie algebra of all graded endomorphisms of $A_{\leq 2m}$, which annihilate these tensors --- that is, the Lie algebra of all graded Poisson derivations of $A$ (the proof of this claim is analogous to \cite[Section 10.7, Corollary]{Hum12}).

By \cite[Proposition 2.14]{Los16}, all Poisson derivations of $A$ are Hamiltonian (see \cite[Lemma 2.9]{ES20} for an alternative analytic proof). By definition, the Lie algebra $\fl_0$ acts on $X$ by graded Hamiltonian derivations and any Hamiltonian derivation is of this form. That is, we have a surjective homomorphism of Lie algebras $\fl_0 \rightarrow \g$.



The kernel of this homomorphism is the intersection of the Poisson center of $A$ with $A_{\bold d}$. Hence, its injectivity would follow if we show that the Poisson center of $A$ is $A_0$.

Indeed, suppose $f \in A$ is in the Poisson center. On the open symplectic leaf of $X$, the Poisson bivector gives a nondegenerate pairing of cotangent and tangent bundles, hence $df = 0$ on the open leaf. Thus, $f$ is constant on the open dense leaf, and thus constant on all~$X$. 
\end{proof}

Since $\{0\}$ is a symplectic leaf, $T^*_0 X$ is naturally a Lie algebra.
\begin{cor} \label{cotangent space at 0 to g}
$T^*_0 X$ is a nonnegatively graded finite-dimensional Lie algebra. Its zeroth graded component is isomorphic to $\g$, so there is a surjection $T^*_0 X \twoheadrightarrow \g$ with nilpotent kernel; in particular, one has an isomorphism of reductive parts $(T^*_0 X)^\red \simeq \g^\red$.
\end{cor}
\begin{proof}
By our assumption, $A_{> 0} = A_{\geq \bold d}$. Hence $T^*_0 X = A_{\geq \bold d} / (A_{\geq \bold d})^2$, and the claim follows from Lemma \ref{A_d = g}.
\end{proof}

In almost all interesting examples, $\g = \g^\red$, though it is not always the case; see the discussion after Theorem 3.15 in \cite{ES20}.

\subsection{Definition of the category} \label{subsec: definition of the category}
Consider the cotangent Poisson Lie algebroid $\Om_X$ (see Example~\ref{example of lie algeboroids}~\ref{example: cotangent algebroid}), determined by the Poisson structure on $X$. 
The embedding of Lie algebras $\g \rightarrow \fl$ of Lemma~\ref{A_d = g} produces a homomorphism $\cO_X \T \g \rightarrow \Omega_X$ of Lie algebroids on $X$, which is given by $f \T g \mapsto fdg$ (here $g \in A_{\bold d}$) under identification of Lemma~\ref{A_d = g}. It can be seen as a comoment map for the Hamiltonian action of $G$ on $X$. It follows that one can consider the Harish-Chandra pair $(G, \Omega_X)$.

Recall the contracting $\torus_\hbar$-action on $X$. Denote $\Lie \torus_\hbar = \bC \hbar$, for a formal variable $\hbar$. $\torus_\hbar$-action induces the $\bC \hbar$-action on $X$, and we can form the Lie algebroid $\Omega_X \oplus \cO_X \hbar$ (here $\cO_X \hbar$ means a trivial rank 1 sheaf, whose sections we denote $f \hbar$, $f \in A$). The Lie bracket on it is given by
\[
[\hbar, dg] = (i - {\bold d})dg
\]
for $g \in A_i$.
Then we naturally have HC-pairs $(\torus_\hbar, \Omega_X \oplus \cO_X \hbar)$ and $(G \times \torus_\hbar, \Omega_X \oplus \cO_X \hbar)$ on~$X$ (this construction of ``enlarging'' Lie algebroid by adding a Lie algebra is the same as in \cite[Lemma~1.8.6]{BB93}). Sometimes, abusing notation, we denote these HC pairs by $(\torus_\hbar, \Om_X)$ and $(G \times \torus_\hbar, \Om_X)$.

By definition, the orbits of $\Omega_X$ are the symplectic leaves of $X$. It is proved in \cite[Section~3]{Kal06} that $X$ has a finite number of those. Any symplectic leaf is obviously even-dimensional (because it is symplectic).

By \cite[Proposition 1.3]{Bea00}, $X$ is Gorenstein. Hence, $\cO_X[\frac{1}{2} \dim X]$ is a dualizing object in $D^b_\coh(X)$. It obviously has an equivariant structure over any HC Lie algebroid, and hence by Lemma \ref{dualizing object iff Forg is dualizing}, derived category of coherent modules over any HC Lie algebroid on $X$ admits a dualizing object.

The two paragraphs above guarantee the validity of the following definition, appealing to the construction of Section \ref{sec: perverse coherent modules over HC Lie algebroids}.

\begin{definition} \label{def: category for sympl sing}
Let $(G, \cL)$ be any choice of Harish-Chandra pair from 
\[
\Omega_X, \qquad (G, \Omega_X), \qquad (\torus_\hbar, \Omega_X \oplus \cO_X \hbar), \qquad (G \times \torus_\hbar, \Omega_X \oplus \cO_X \hbar).
\]
Define the middle perversity by $p(S) = - \frac{1}{2} \dim S$ for any leaf $S$.

We call $\cP^{(G, \cL)}_\coh(X)$ as the category of {\bf perverse coherent} (or perverse Poisson) {\bf sheaves of middle perversity on the symplectic singularity} $X$.
\end{definition}

Note that the categories $\cP^{\Om_X}_\coh(X)$, $\cP^{\torus_\hbar, \Om_X}_\coh(X)$ make sense for not necessarily conical symplectic singularities, and, more generally, for any Poisson varieties with a finite number of symplectic leaves.

Theorem \ref{main theorem on perverse coherent sheaves} applies to any choice of $\cL$ as above. In particular, we have a classification of simples in this category as IC-extensions of simples on a leaf. In the following subsection, we study in more detail what these simples are.

\subsection{On simple modules on a symplectic leaf}
In this subsection, we work with the HC pair $(\bC^\times_\hbar, \Om_X \oplus \bC \hbar)$ (so we consider graded Poisson sheaves), which we denote simply by $(\torus_\hbar, \Om_X)$. Let $\rho$ be its anchor map. Take a symplectic leaf $S$. We are interested in simple $(\left. \torus_\hbar, \Om_X \right|_S)$-modules on $S$. Equivalently, we are interested in the reductive part of the differential Galois group $\Gal^{(\torus_\hbar, \Om_X \vert_S)}(S)$ (see Section~\ref{subsection differential galois group}). Let us introduce some notation.

For a Lie algebra $\mathfrak f$, the category $\mathfrak f\mods$ of its finite-dimensional representations is Tannakian. We denote by $(\exp \mathfrak f)$ its Tannakian pro-algebraic group. There is a full Tannakian subcategory of $\mathfrak f\mods$, generated by the adjoint representation, denoted $\bra \ad_{\mathfrak f} \ket^{\T}$. Its Tannakian group is algebraic (finite-dimensional), we denote it by $\exp (\mathfrak f)^\ad$. $\exp (\mathfrak f)^\ad$ is canonically a quotient of $\exp \mathfrak f$.

If a 2-groupoid with one object $\Gamma$ acts on the category $\mathfrak f\mods$, we denote by $\exp^\Gamma \mathfrak f$ the Tannakian group of the category of equivariant objects $(\mathfrak f\mods)^\Gamma$. Denote by $\Gamma_{\leq 1}$ the 1-truncation.
Suppose 2-morphisms of $\Gamma$ act by trivial automorphisms on the adjoint representation~$\ad_\ff \in \ff\mods$.
Then there is a full Tannakian subcategory $\bra \ad_{\mathfrak f}, (\Vect)^\Gamma \ket^\T$, generated by the adjoint representation and by the subcategory $(\Vect)^\Gamma$
(here $\Vect$ is the subcategory of trivial representations in $\ff\mods$).
We denote its Tannakian group by $\exp^\Gamma (\mathfrak f)^\ad$. 
If $\Gamma_{\leq 1}$ is finite, $\exp^\Gamma (\mathfrak f)^\ad$ is finite-dimensional algebraic. $\exp^\Gamma (\mathfrak f)^\ad$ is canonically a quotient of $\exp^\Gamma \mathfrak f$.

Take a symplectic leaf $S \neq \{0\}$.
Take a closed $x \in S$. Consider its $\torus_\hbar$-orbit $\torus_\hbar x$. Consider the ``stabilizer of $\torus_\hbar x$ in $\Om_X$'': $\h_{\torus_\hbar x} := \rho_x^{-1}(\bC_\hbar x)$ (here $\bC_\hbar x \subset T_x S$).
Equivalently, it is the inertia Lie algebra of the algebroid $\Om_X \oplus \cO_X \hbar$ at $x$.

The result below is our best estimate of the category of modules on a general leaf of a general symplectic singularity.

\begin{thm} \label{thm: on simple modules on a symplectic leaf}
Let $S$ be a symplectic leaf of $X$.
\begin{enumerate}[1)]
    \item\label{case: S neq 0} Suppose $S \neq \{ 0 \}$.
There are faithfully flat surjections of pro-algebraic groups
\begin{equation} \label{eq: exp to gal to exp}
\exp^{\pi_{\leq 2}(\bP S)}(\h_{\torus_\hbar x}) \twoheadrightarrow \Gal^{(\torus_\hbar, \Om_X)} (S) \twoheadrightarrow \exp^{\pi_{\leq 2}(\bP S)}(\h_{\torus_\hbar x})^\ad,
\end{equation}
where $\bP S = S / \torus_\hbar$.

\item\label{case: S = 0} Suppose $S = \{0\}$. Then simple objects of $\cP_\coh^{(\torus_\hbar, \Om_X)}(X)$, supported at 0, are in correspondence with simple $\g$-modules up to a grading shift. In other words, \[\Gal^{(\torus_\hbar, \Om_X)}(S)^\red = (\exp \g)^\red \times \torus_\hbar.\]
\end{enumerate}
\end{thm}

\begin{proof}
We first prove part~\ref{case: S neq 0}.

Consider the normalization of $\nu: \wti S \rightarrow \ol S$ of the closure $\ol S$ of $S$. $\wti S$ is an affine conical symplectic singularity: it follows from \cite[Theorem~2.5]{Kal06} that it is a symplectic singularity, and construction of the conical action is the same as in \cite[Lemma~2.5]{Los21}; we denote this conical action by $\torus_\hbar \curvearrowright \wti S$, so $\nu$ is $\torus_\hbar$-equivariant.

$\nu$ is a Poisson isomorphism over $S$, meaning that $\nu^{-1}(S)$ lies in the open leaf of $\wti S$.
Normalization is a finite morphism, hence \[\dim (\wti S \setminus \nu^{-1}(S)) = \dim (\ol S \setminus S)  \leq \dim \ol S - 2,\] so $\nu^{-1}(S)$ is open with complementary of codimension $\geq 2$ in $\wti S$.
In what follows, we identify $S$ and $\nu^{-1}(S)$.

$\nu^* ( \Om_X \vert_{\ol S})$ is naturally a Lie algebroid on $\wti S$ by \cite[Remark~2.4.3.(2)]{Kae98}. We denote it $\Om_X \vert_{\wti S}$, and consider the HC Lie algebroid $(\bC^\times_\hbar, \Om_X \vert_{\wti S})$ on $\wti S$.

Consider the projectivization of $\wti S$, $\bP \wti S := \left. (\wti S \setminus \{0\}) \right/ \torus_\hbar$.
From normality of $\wti S$, one can easily deduce that every principal open $\Spec (\cO(\wti S)_{(g)})$ (for homogeneous $g \in \cO(\wti S)$) is normal, hence $\bP \wti S$ is normal.

By \cite[Lemma~1.8.7]{BB93}, the Lie algebroid $\Om_X \vert_{\wti S}$ descends to $\bP \wti S$, denote it by $\bP(\Om_X \vert_{\wti S})$, and we have equivalences of categories 
\begin{equation*}
\Coh^{(\torus_\hbar, \Om_X\vert_{\wti S})}\bigl(\wti S\bigr) \simeq \Coh^{\bP(\Om_X \vert_{\wti S})} \bigl(\bP \wti S\bigr), \qquad \Coh^{(\torus_\hbar, \Om_X\vert_{S})}\bigl( S \bigr) \simeq \Coh^{\bP(\Om_X \vert_{S})} \bigl(\bP S \bigr).
\end{equation*}

In summary of all of the above, we are in the setup of Proposition~\ref{prop: gaga for open of codim 2}. Namely, we can apply this Proposition to projective $\bP \wti S$ and its open $\bP S$ with complementary of codimension $\geq 2$. Corollary~\ref{cor: analytic galois to algebraic is faithfully flat gaga} tells that analytification on $\bP S$ induces a faithfully flat surjection to the differential Galois group $\Gal^{\bP(\Om_{X} \vert_S)}(\bP S)$ from its analytic version, which, by Theorem~\ref{thm: L-modules is equivariantization} is isomorphic to $\exp^{\pi_{\leq 2}(\bP S)}(\h_{\torus_\hbar x})$. So we established the surjection 
\begin{equation} \label{eq: exp to gal}
\exp^{\pi_{\leq 2}(\bP S)}(\h_{\torus_\hbar x}) \twoheadrightarrow \Gal^{(\torus_\hbar, \Om_X)} (S).
\end{equation}

To establish the second arrow in~\eqref{eq: exp to gal to exp}, we need to show that the essential image of the analytification functor on $\Coh^{\bP(\Om_X \vert_S)}(\bP S)$ (whose Tannaka-dual is~\eqref{eq: exp to gal}) contains $\Vect^{\pi_1(\bP S )}$ and the adjoint representation of $\h_{\torus_\hbar}$. 

Indeed, $\Vect^{\pi_1(\bP S )}$ is identified with the category of $\cT_{(\bP S)^\an}$-modules, given the structure of $\bP(\Om_X\vert_S)^\an$-modules by restriction along the anchor map. Analytification is essentially surjective on coherent $\cT_{\bP S}$-modules (local systems) by the Riemann--Hilbert correspondence, hence the claim.
Note that by Remark~\ref{rem: action of pi_2 on vect is trivial}, the action of $\pi_2(\bP S)$ on the identity endofunctor of $\Vect$ and $\Vect^{\pi_1(\bP S)}$ is trivial.

The adjoint representation $\h_{\torus_\hbar x}$ lies in the essential image of analytification simply because it comes from the $\bP(\Om_X \vert_S)$-module $(\ker \rho)$ (see Example~\ref{example: adjoint modules}~\ref{example: adjoint action on ker rho}), which is obviously algebraic.

Now we prove~\ref{case: S = 0}.

Simple modules at 0 are the same as simple modules over the HC pair $(\torus_\hbar, \Om_X \vert_0) = (\torus_\hbar, T^*_0 X)$. Now the claim follows from Corollary~\ref{cotangent space at 0 to g}.
\end{proof}

On the open leaf, the category in question is the category of weakly-equivariant $\cO$-coherent D-modules. We expect that from the Riemann--Hilbert correspondence, one can deduce that analytification is essentially surjective on this category; in this case, the argument from our proof above would imply that the first arrow in~\eqref{eq: exp to gal to exp} is an isomorphism for it. See~\cite[Lemma~2.12]{Los21}.

\begin{rem}
Note that the profinite completion of $\pi_1(S)$ is finite by the main result of~\cite{Nam13}. One can show that the homomorphism of \'etale fundamental groups, induced by $S \rightarrow \bP S$ is surjective, hence the profinite completion $\wh \pi_1(\bP S)$ of $\pi_1(\bP S )$ is finite. One can expect that the action of $\pi_1(\bP S)$ on $\h_{\torus_\hbar x}\mods$ factors through $\wh \pi_1(\bP S)$ (compare with a theorem of Grothendieck~\cite{Gro70}). If so, one can formally deduce from~\cite{Gro70} that the equivariantizations $(\h_{\torus_\hbar x}\mods)^{\pi_1(\bP S)}$ and $(\h_{\torus_\hbar x}\mods)^{\wh \pi_1(\bP S)}$ are equivalent.
\end{rem}

\subsection{Remarks and further questions}
In this subsection, we make a few remarks about expected properties of the described general construction. See also Section~\ref{intro subsubsec: further questions}.


\begin{rem} \label{rem: !-restriction to slice}
It is a theorem of Kaledin that any symplectic leaf of a symplectic singularity admits a formal slice at any point, which is also a (formal) symplectic singularity \cite[Theorem~2.3]{Kal06}. 
Recently, it was proved in great generality by Namikawa--Odaka \cite{NO25} that this slice is conical. It is also known in many examples that such slice exists in \'etale topology (not just formally), see e.g. \cite[Section~7]{KT21}.

We expect that $!$-restriction to a slice at any point should be $t$-exact with respect to perverse coherent $t$-structures. See Remark~\ref{rem: pullback of poisson sheaves to slice} below.

In analytic topology, we may also have the following. Consider the union of symplectic leaves, whose closure contains the fixed leaf $S$. Then on a small chart, which is Poisson-isomorphic to the product of a disk on $S$ and a slice to $S$, a perverse coherent sheaf is the same as just a perverse coherent sheaf on the slice. This allows to consider the category of perverse coherent sheaves as a local system of categories over $S$. Then, similarly to Section~\ref{subsec: modules over transitive is equivariantization}, one may consider an action of the homotopy groupoid of $S$ on the category of perverse coherent sheaves on slice, and equivariantizations.
\end{rem}


\begin{rem}
    Let $\tilde G \twoheadrightarrow G$ be a morphism of algebraic groups. Then one can consider the HC pairs $(\tilde G, \Om_X)$, $(\tilde G \times \bC_\hbar^\times, \Om_X \oplus \cO_X \hbar)$, acting on $X$, and carry out the same construction for this pair (instead of $(G \times \bC_\hbar^\times, \Om_X \oplus \cO_X \hbar)$).

    For instance, if $X = \cN \subset \g$ is the nilpotent cone, the group of graded Poisson automorphisms is the adjoint group $G^{ad}$, while one may want to consider its covering $G$ (e.g. the simply connected group $G^{sc}$) and study the perverse basis with this equivariance, see Section~\ref{subsec: nilcone}.

    If $X$ is the quiver variety $\fM_{Q}(\bv, \bw)$ for an (oriented) tree $Q$, then the group of graded Poisson automorphisms is a quotient of $PGL_\bw = (\prod_i GL_{\bw_i} )/ \bC^\times$ (see \cite[Section~9.5]{BLPW14}), while one may want to replace it with the extension $GL_\bw = \prod_i GL_{\bw_i}$, considered in e.g.~\cite{Nak01a}.
\end{rem}

\begin{rem}
As we pointed out in Section~\ref{subseq: HC lie algebroids}, we consider only strongly equivariant modules over HC Lie algebroids in this paper. However, it may be interesting to investigate weakly-equivariant perverse coherent sheaves on symplectic singularities.
\end{rem}


\section{Examples} \label{sec: examples}
In this section, we treat the introduced general notion for certain particular symplectic singularities.

\subsection{Nilpotent cone} \label{subsec: nilcone}
Let $\g$ be a simple Lie algebra, $G$ some algebraic group, whose Lie algebra is $\g$, and $\cN \subset \g$ be the nilpotent cone. In this case, the original construction of \cite{Bez00, AB10} can be applied, and in this way one gets the perverse basis in $K^{G \times \torus_\hbar}(\cN)$. It follows from \cite{Bez06a, Bez09} that this basis is a part of the Kazhdan--Lusztig canonical basis in the affine Hecke algebra for the Langlands-dual group (see also \cite{Ost00}).

On the other hand, we have the construction of Section \ref{Section: symplectic singularities}.
The aim of this subsection is to show that it gives the same basis of the same space in this case.
Thus, our construction of the basis in Section~\ref{Section: symplectic singularities} can indeed be seen as a generalization of the known case of $\cN$ to the case of an arbitrary conical symplectic singularity.

\begin{thm} \label{perverse bases for nilcone are the same}
The functor $F: D^{b, (G, \Omega_\cN)}_\coh (\cN) \rightarrow D^{b, G}_\coh(\cN)$, which forgets the $\Omega_\cN$-action, is t-exact with respect to perverse t-structure on both sides. The corresponding functor $\cP^{(G, \Omega_\cN)}_\coh(\cN) \rightarrow \cP^G_\coh (\cN)$ maps simple objects to simple objects, and defines a bijection between isomorphism classes of simples.
In particular, the natural map 
\[
K^{(G, \Omega_\cN)}(\cN) \rightarrow K^G(\cN)\] is isomorphism, preserving perverse bases on both sides.

The same holds in the graded case. 

For $G = G^{\mathrm{sc}}$ simply connected, there is an isomorphism
\begin{equation} \label{eq: iso of poisson end equi K-theory of nilcone}
K^{(\torus_\hbar, \Om_\cN)}(\cN) \simeq K^{G \times 
\torus_\hbar}(\cN),
\end{equation}
preserving perverse bases.
\end{thm}


\begin{proof}
Clearly, $(G, \Omega_\cN)$-orbits are the same as $G$-orbits.
Perversity of a sheaf is defined in terms of cohomological properties of restrictions to orbits, without appealing to equivariance. Thus, $F$ is indeed t-exact.

Take an orbit $j : O \rightarrow \cN$. We claim that $F \circ j_{!*} = j_{!*} \circ F$. 
Indeed, IC-extensions in Section~\ref{subsection: IC-extension} and in \cite[Section~3.2]{Bez00} are both defined as inverse to the restriction from the category $\cP_{!*}(O)$, which is defined in purely cohomological terms, independent of equivariance. Hence, $F$ intertwines these categories, and the claim follows.

Thus, to finish the proof, it is sufficient to show that for any fixed orbit $O$ the functor $F_O: \Coh^{(G,  \Omega_\cN\vert_O)}(O) \rightarrow \Coh^G(O)$ between abelian categories maps simples to simples and defines a bijection on isomorphism classes of those. From the tautological Poisson embedding $\cN \hookrightarrow \g$, we see that there is a surjection $\cO_\cN \T \g \twoheadrightarrow \Omega_\cN$ of Lie algebroids on $\cN$. Restricting to $O$, it becomes clear that $F_O: \Coh^{(G,  \Omega_\cN\vert_O)}(O) \rightarrow \Coh^{(G, \g)}(O) = \Coh^{G}(O)$ defines an injection on classes of simples.

Finally, we claim that the action of $(G, \g)$ on any simple module on $O$ factors through $(G, \Om_\cN \vert_O)$. Take a simple $(G, \g)$-module $V$, let $a: \cO_O \T \g \rightarrow \cEnd_\bC V$ be the action morphism, and let $\phi: \cO_O \T \g \rightarrow \Om_\cN \vert_O$ be the surjection. We need to show that $\ker \phi \subseteq \ker a$. 
This can be checked locally. 
Pick $e \in O$ and consider the formal neighborhood $O^{\wedge e}$. We have a $\cO_O^{\wedge e} \T \g$-module $V^{\wedge e}$ and the surjection $\phi^{\wedge e}: \cO_O^{\wedge e} \T \g \twoheadrightarrow \Om_\cN\vert_O^{\wedge e}$.
For any transitive Lie algebroid $(\cL, \rho)$ on $O^{\wedge e}$, \cite[Theorem A.7.3]{Kap07} tells that the category of $\cL$-modules is equivalent to the category of $(\ker_e \rho)$-modules by means of restriction to $e$. 
Applying it to $\cO_O^{\wedge e} \T \g$ and $\Om_\cN \vert_O^{ \wedge e}$, one sees that it is sufficient to show that $\ker \ol \phi_e \subset \ker \ol a_e$, where $\ol \phi_e: \g_e \rightarrow \fh_e$, $\ol a_e: \g_e \rightarrow \End_\bC (V_e)$ are fibers of morphisms of inertia bundles (here $\h$ is the inertia bundle of $\Om_\cN\vert_O$). 

Note that $V$ is a simple $G$-equivariant vector bundle on $O$; equivalently, $V_e$ is a simple $G_e$-module. In particular, $\ker \ol a_e$ contains the Lie algebra $\g_e^\mathrm{u}$ of the unipotent radical $G_e^{\mathrm{u}}$ of~$G_e$.

Now we claim that $\ker \ol \phi_e \subseteq \g_e^\mathrm{u}$. 
Indeed, we have $\Om_\cN \vert_e = T^*_e S_e \oplus T_e O$, where $S_e$ is the Slodowy slice at $e$, and there is an isomorphism of Lie algebras $T^*_e S_e \simeq \h_e$. It is well known that the Lie algebra of graded Poisson derivations of $S_e$ is the Lie algebra $\g_e^\red$ of the reductive part $G_e^\red$. 
So we have the composition $\g_e \xrightarrow{\ol \phi_e} \h_e \simeq T_e^*S_e \twoheadrightarrow \g_e^\red$, where the last arrow is given by Corollary \ref{cotangent space at 0 to g}. Since $\g_e^\red \simeq \g_e / \g_e^{\mathrm{u}}$, it follows that $\ker \ol \phi_e \subseteq \g_e^\textrm{u}$.

We have shown that $\ker \ol \phi_e \subseteq \g_e^\textrm{u} \subseteq \ker \ol a_e$, and the claim is proved.

The proof of the graded case is the same. 

Isomorphism~\eqref{eq: iso of poisson end equi K-theory of nilcone} is justified as follows: 
a graded coherent $\Om_\cN$- module over $\bC[\cN]$ is a direct sum of finite-dimensional graded components, and each component is stable under the $\g$-action (which comes from the map $\cO_\cN \T \g \rightarrow \Om_\cN$). 
Hence, the condition of being integrable along $G^{\mathrm{sc}}$ holds automatically for graded modules,
and $K^{(\torus_\hbar, \Om_\cN)}(\cN) \simeq K^{(\torus_\hbar \times G, \Om_\cN)}(\cN) \simeq  K^{G \times \torus_\hbar}(\cN)$. 
\end{proof}

\begin{example}
Let us illustrate Theorem~\ref{perverse bases for nilcone are the same} for two extreme orbits: the open and the closed. For simplicity, take the simply connected group $G^\mathrm{sc}$.

Consider simple $G^\mathrm{sc}$-equivariant sheaves on the open orbit $O^\mathrm{reg}$ in $\cN$. It is well known that the connected component of the stabilizer of a point on $O^\mathrm{reg}$ in $G^\mathrm{sc}$ is unipotent. Hence, on any irreducible representation, the action factors through the quotient by the connected component of identity, which means that the corresponding vector bundle is a local system, and the induced action of $\g$ indeed factors through $\Om_{\cN} \vert_{O^\mathrm{reg}} \simeq \Om_{O^\mathrm{reg}}$, as Theorem~\ref{perverse bases for nilcone are the same} predicts.
At the same time, we see that the action of $\g$ on a module that is not simple need not factor through $\Om_\cN$.

Now consider the closed orbit $\{0\} \in \cN$. On it, the category of $G^\mathrm{sc}$-representation indeed coincides with the category of modules over $\Om_{\cN}\vert_0 \simeq T^*_0 \cN \simeq \g$, also in agreement with Theorem~\ref{perverse bases for nilcone are the same}.
\end{example}

It would be interesting to investigate an analogue of Theorem~\ref{perverse bases for nilcone are the same} for the case of Kato's exotic nilpotent cone, see \cite{Nan13}.



\subsection{Affine Grassmannian slice} \label{subsec: affine grassmannian slices}

The category of perverse coherent sheaves on the affine Grassmannian was studied in \cite{BFM05, CW19, FF21, Dum24}. In particular, in \cite{FF21} it is proved that in type A, classes of simple perverse coherent sheaves give Lusztig's dual canonical basis. 
In arbitrary type, we conjecture that under validity of \cite[Conjecture~1.10]{CW19}, the basis of simples in equivariant K-theory is Qin's common triangular basis in the corresponding quantum cluster algebra \cite{Qin17}. 
This is compatible with the result of \cite{FF21} in type A, see \cite{Qin24}. Note that outside of type A, it is not yet known whether the common triangular basis exists in this quantum cluster algebra.

In this section, we prove that this basis arises as a particular case of our general construction, applied to a slice in the affine Grassmannian. Combining with results of Section~\ref{subsec: nilcone}, one can say that in a certain sense, our construction is a generalization of both mainly studied appearances of perverse coherent bases: in the nilpotent cone and in the affine Grassmannian.

Let us fix some notation. We assume that $G = G^{\mathrm{sc}}$ is a simply connected simple group, and denote by $G^\ad$ its adjoint form. We consider the thick affine Grassmannian $\Gr = \Gr_{G^\mathrm{sc}} = G^\mathrm{sc}((t^{-1})) / G^\mathrm{sc}[t]$ of $G^\mathrm{sc}$. For a dominant coweight $\la$ of $G^\mathrm{sc}$, 
we denote $t^\la \in \Gr$ the corresponding point in $\Gr$, and by $\Gr^\la$ its $G[t]$-orbit. We also have the group $G_1[[t{^{-1}}]]$, the kernel of the evaluation $t^{-1} \mapsto 0$ projection $G[[t^{-1}]] \rightarrow G$. Denote $\cW_\mu = G_1[[t^{-1}]] \cdot t^\mu$. We consider the transversal slice $\cW^\la_\mu = \ol \Gr^\la \cap \cW_\mu$. It is a conical symplectic singularity, whose symplectic leaves are $\Gr^\nu \cap \cW^\la_\nu$ for $0 \leq \nu \leq \la$, see \cite{KWWY14} for details ($\leq$ means the dominance partial order on coweights here and further). There is the loop rotation torus action on $\Gr$, we denote it by $\torus_\hbar \curvearrowright \Gr$.

Let $G^\ad_{\mu}$ be the centralizer of a cocharacter $\mu$ in $G^\ad$. This is a reductive group. 
It is easy to see that this group acts on $\cW^\la_\mu$ by graded Poisson automorphisms\footnote{It is claimed without proof in \cite[9.6(i)]{BLPW14} that the group of graded Poisson automorphisms of $\cW^\la_\mu$ is $G^\ad_{\la, \mu}$ --- the common centralizer of $\la$ and $\mu$. This is false as stated. For example, if $\mu = 0$, the whole $G^\ad$ acts on $\cW^\la_0$, and it does not need to centralize $\la$. Another case when it is false is when $\la = \mu$, since $\cW^\mu_\mu = \mathrm{pt}$, and the action of $G^\ad_\mu$ is not faithful. The purpose of Lemma \ref{lem: centralizer acts faithfully on slice} is to show that the last problem does not occur under a reasonable assumption. 
We expect that if its action is faithful, $G^\ad_{\mu}$ is the group of graded Poisson automorphisms of $\cW^\la_\mu$, but we do not prove it here.}.

\begin{lem} \label{lem: centralizer acts faithfully on slice}
Suppose $\mu < \la$ and there is $\la'$ such that  $\mu < \la' \leq \la$, and $ \la' - \mu$ is dominant. Then the action of $G^\ad_{\mu}$ on $\cW^\la_\mu$ is faithful. Hence, $G^\ad_{\mu}$ is a subgroup of the group of graded Poisson automorphisms of $\cW^\la_\mu$.
\end{lem}
\begin{proof}
$\cW^{\la'}_\mu$ is a closed subscheme of $\cW^\la_\mu$, so it is sufficient to check the faithfulness of the action on $\cW^{\la'}_\mu$. 
That is what we show below.

The action of a group on a scheme is faithful if an only if it is faithful on an open dense invariant subscheme, so it is a birational invariant. Multiplication by $t^{\mu}$ inside the Grassmannian gives a $G^\ad_{\mu}$-equivariant morphism $\cW_0^{\la' - \mu} \rightarrow \cW^{\la'}_\mu$ (see \cite[Section~2.5]{KWWY14}). This is a particular case of the slice multiplication morphism, hence it is birational (see \cite[Section~2.(vi)]{BFN19}, \cite[Remark~5.8]{KP21}). 
So it is sufficient to check that the action of $G^\ad_\mu$ on $\cW_0^{\la' - \mu}$ is faithful. We check that even the action of the whole $G^\ad$ on it is faithful (recall that $\cW_0^{\la' - \mu}$ is $G^\ad$-invariant). Indeed, the subgroup of $G^\ad$, which acts trivially on the whole variety is normal. Since $G^\ad$ is simple and its action on $\cW_0^{\la' - \mu}$ is not trivial, this subgroup is trivial, as required.
\end{proof}

There is an action of $G(\cO) \simeq G[[t]]$ on $\ol \Gr^\la$; it induces an action of the Lie algebra $\g[[t]]$ by vector fields on the same space. This action restricts to a $\g[[t]]$-action on open $\cW^\la_0 \subset \ol \Gr^\la$. Consider the category of graded ($\torus_\hbar$-equivariant) $\g[[t]]$-equivariant coherent sheaves on $\cW^\la_0$, denoted $\Coh^{(\torus_\hbar, \g[[t]])}(\cW^\la_0)$.

\begin{prop} \label{prop: from schubert to slice}
The restriction induces an equivalence of categories:
\begin{equation}  \label{from schubert to slice}
\Coh^{\torus_\hbar \ltimes G[[t]]} (\ol \Gr^{\la}) \xrightarrow{\sim}    \Coh^{(\bC^\times_\hbar, \g[[t]])}(\cW^\la_0).
\end{equation}
The corresponding derived equivalence is t-exact with respect to perverse t-structures.
\end{prop}
\begin{proof}
Consider the action Lie groupoid $(\torus_\hbar \ltimes G(\cO)) \times \ol \Gr^\la \rightrightarrows \ol \Gr^\la$. Restrict it to the open subscheme $\cW^\la_0$, get a Lie groupoid, which we denote $\cG \rightrightarrows \cW^\la_0$. Since $\cW^\la_0$ intersects every $\torus_\hbar \ltimes G(\cO)$-orbit in $\ol \Gr^\la$, there is an isomorphism of quotient stacks $\Big[\left. \ol \Gr^\la \right/ \torus_\hbar \ltimes G(\cO)\Big] = [\left. \cW^\la_0 \right/ \cG]$. In particular, the category of $\cG$-equivariant coherent sheaves on $\cW^\la_0$ is equivalent to the category of $\torus \ltimes G(\cO)$-equivariant coherent sheaves on $\ol \Gr^\la$.

On the other hand, due to $\torus_\hbar$-equivariance, the category $\Coh^{\torus_\hbar, \g[[t]]}(\cW^\la_0)$ is equivalent to the same category on $(\cW^\la_0)^{\wedge t^0}$ --- the formal completion of $\cW^\la_0$ at $t^0 \in \cW^\la_0$: 
\begin{equation*}
\Coh^{\torus_\hbar, \g[[t]]}(\cW^\la_0) \simeq \Coh^{\torus_\hbar, \g[[t]]}((\cW^\la_0)^{\wedge t^0}).
\end{equation*}
Consider the $n$-th formal neighborhood $(\cW^\la_0)^{\wedge_n t^0}$. $\bC[(\cW^\la_0)^{\wedge_n t^0}]$ is a finite-dimensional graded algebra, and $\Coh^{\torus_\hbar, \g[[t]]}((\cW^\la_0)^{\wedge_n t^0})$ is the category of $\g[[t]]$-equivariant graded finitely generated modules over it. It is clear that it is equivalent to the category of $G(\cO)$-equivariant graded finitely generated modules (we use that $G$ is simply connected here). So, passing to limit, we obtain an equivalence
\begin{equation} \label{eq: from lie alg to group on formal}
\Coh^{\torus_\hbar, \g[[t]]}((\cW^\la_0)^{\wedge t^0}) \simeq \Coh^{G(\cO) \rtimes \torus_\hbar}((\cW^\la_0)^{\wedge t^0}).
\end{equation}
Consider $\cG^{\wedge t^0}$ --- the restriction of $\cG$ to $(\cW^\la_0)^{\wedge t^0}$. Since $t^0$ is $G(\cO) \rtimes \torus_\hbar$-fixed, we see that 
\begin{equation} \label{eq: from group to groupoid on formal}
\Coh^{G(\cO) \rtimes \torus_\hbar}((\cW^\la_0)^{\wedge t^0}) \simeq \Coh^{\cG^{\wedge t^0}}((\cW^\la_0)^{\wedge t^0})
\end{equation}
(the composition of~\eqref{eq: from lie alg to group on formal},~\eqref{eq: from group to groupoid on formal} is a particular case of a general phenomenon, that on a formal scheme, supported on a point, the data of a Lie algebroid is the same as the data of a formal Lie groupoid, see e.g. \cite[Appendix~A]{Kap07}).

Using $\torus_\hbar$-equivariance again, we return from $\cG$-equivariant sheaves on completion to those on the initial variety. Summing up all of the above, we have
\begin{multline*}
\Coh^{(\bC^\times_\hbar, \g[[t]])}(\cW^\la_0) \simeq \Coh^{(\bC^\times_\hbar, \g[[t]])}((\cW^\la_0)^{\wedge t^0}) \simeq \Coh^{\cG^{\wedge t^0}} ((\cW^\la_0)^{\wedge t^0}) \\
\simeq \Coh^{\cG} (\cW^\la_0) \simeq \Coh^{\torus_\hbar \ltimes G[[t]]} (\ol \Gr^{\la}).
\end{multline*}

The second claim about perverse t-exactness is obvious from the above and definitions. 
\end{proof}

Note also that there are equivalences 
\begin{equation} \label{eq: double and single brackets are the same}
\Coh^{G[[t]] \rtimes \torus_\hbar}(\ol \Gr^\la) \simeq \Coh^{G[t] \rtimes \torus_\hbar}(\ol \Gr^\la), \qquad \Coh^{(\torus_\hbar, \g[[t]]}(\ol \Gr^\la) \simeq \Coh^{(\torus_\hbar, \g[t]}(\ol \Gr^\la),
\end{equation}
since on any coherent sheaf, the action factors through quotient by some power of $t$. Similarly for $\cW^\la_0$. From now on, we may identify these categories.

Note that the $\g[t]$-action on $\cW_0 \simeq G_1[[t^{-1}]]$ is nothing else but the infinitesimal dressing action for the Manin triple $(\g[t], t^{-1}\g[[t^{-1}]], \g((t^{-1})))$ (see e.g. \cite[Section~2]{LW90}). That is, for the purpose of defining the Poisson--Lie structure on $G_1[[t^{-1}]]$, one can think of the cotangent sheaf $\Om_{G_1[[t^{-1}]]}$ as of $\cO_{G_1[[t^{-1}]]} \otimes \g[t]$, equipped with the Lie algebroid structure by means of the dressing action ($G_1[[t^{-1}]]$ is infinite-dimensional, and in its cotangent bundle $G_1[[t^{-1}]] \times \g_1[[t^{-1}]]^*$ the dual should be understood as the dual element of the Manin triple). Restricting to the closure of a Poisson leaf $\cW^\la_0$ we have a surjection of Lie algebroids
\begin{equation} \label{surjection g[t] to cotangent on slice}
\cO_{\cW^\la_0} \T \g[t] \simeq  \Om_{\cW_0}\Big| _{\cW^\la_0} \twoheadrightarrow \Om_{\cW_0^\la}.
\end{equation}

Until the end of this subsection, we write $\Om$ for $\Om_{\cW_0^\la}$.

\begin{thm} \label{thm: basis for slice in affine gr}
The functor 
\[F: D^{b, (\bC^\times_\hbar, \Om)}_\coh (\cW_0^\la) \rightarrow D_\coh^{b, (\bC^\times_\hbar, \g[t])} (\cW_0^\la) \simeq D_\coh^{b, \torus_\hbar \ltimes G[[t]]}(\ol \Gr^\la),\] induced from \eqref{surjection g[t] to cotangent on slice}, \eqref{from schubert to slice}, \eqref{eq: double and single brackets are the same} is t-exact with respect to perverse t-structures on both sides. The corresponding functor $\cP^{(\bC^\times_\hbar, \Om)}_\coh (\cW_0^\la) \rightarrow \cP_\coh^{(\bC^\times_\hbar, \g[t])} (\cW_0^\la)$ maps simple objects to simple objects and defines an injection on the isomorphism classes of those;
it defines a bijection between isomorphism classes of simples with support $\cW_0^\nu$ for $\nu$ such that there is $\nu^\prime$ s.t. $\la > \nu' \geq \nu$ and $\la - \nu'$ is dominant.

In particular, for any $\nu' < \la$ s.t. $\la - \nu'$ is dominant, we get a diagram of K-groups 
\[\begin{tikzcd}
	{K^{(\mathbb C^\times_\hbar, \Omega_{\mathcal W^\lambda_0})}(\mathcal W^\lambda_0)} & { K^{\mathbb C^\times_\hbar \ltimes G(\mathcal O)}(\overline{\mathrm{Gr}}^\lambda)} \\
	{K^{(\mathbb C^\times_\hbar, \Omega_{\mathcal W^\lambda_0})}(\mathcal W^{\nu'}_0)} & {K^{\mathbb C^\times_\hbar \ltimes G(\mathcal O) }(\overline{\mathrm{Gr}}^{\nu'})}
	\arrow[hook, from=1-1, to=1-2]
	\arrow[hook, from=2-1, to=1-1]
	\arrow["\simeq", from=2-1, to=2-2]
	\arrow[hook, from=2-2, to=1-2]
\end{tikzcd}\]
in which every arrow respects perverse bases.
\end{thm}

\begin{proof}
The proof basically repeats the proof of Theorem \ref{perverse bases for nilcone are the same}, the main difference being the usage of the Lie groupoid $\cG$ appearing in the proof of Proposition~\ref{prop: from schubert to slice}, instead of a Lie group action.

The fact that $F$ is t-exact and commutes with IC-extension from any leaf follows by the same argument as in the proof of Theorem~\ref{perverse bases for nilcone are the same}.

Pick a leaf $O_\nu = G(\cO).t^\nu \cap \cW_0$ for $0 \leq \nu < \la$. The functor $F_\nu: D^{b, (\bC^\times_\hbar, \Om)}_\coh (O_\nu) \rightarrow D_\coh^{b, (\bC^\times_\hbar, \g[t])} (O_\nu)$ maps distinct simples to distinct simples, since \eqref{surjection g[t] to cotangent on slice} is surjective. It is left to show that if $\nu$ is s.t. there is $\nu'$ as in the assumption, the action of $\g[t] \oplus \bC \hbar$ factors through $\Om \oplus \cO_{\cW^\la_0} \hbar$ on any simple $(\g[t] \oplus \bC \hbar)$-equivariant sheaf $V$ on $O_\nu$.

Denote $a: \cO_{O_\nu} \T (\g[t] \oplus \bC \hbar) \rightarrow \cEnd_\bC V$, $\phi: \cO_{O_\nu} \T (\g[t] \oplus \bC \hbar) \rightarrow \Om \vert_{O_\nu} \oplus \cO_{O_\nu} \hbar$. Denote also by $\h$ the inertia bundle of $(\Om \oplus \cO_{\cW^\la_0} \hbar)\vert_{O_\nu}$.
By the same reasoning as in the proof of Theorem~\ref{perverse bases for nilcone are the same} (using \cite[Theorem~A.7.3]{Kap07}), 
it is sufficient to show that for any point $e \in O_\nu$, we have $\ker \ol \phi_e \subseteq \ker \ol a_e$, where $\ol a_e: (\g[t]\oplus \bC \hbar)_e \rightarrow \End_\bC V_e$, $\ol \phi_e: (\g[t] \oplus \bC \hbar)_e \rightarrow \h_e$. 
Conjugating by an element of $G[t]$, we may assume $e = t^{\nu}$ (although $t^\nu$ does not lie in $O_\nu$, all vector bundles $\torus_\hbar \ltimes G(\cO)$-equivariantly prolong to $\ol \Gr^\la$). Also, both morphisms $\ol a_e$ and $\ol \phi_e$ factor through $\g[t] / t^N \oplus \bC \hbar$ for some $N$, and we from now on assume that $\ol a_e$ and $\ol \phi_e$ have domain $(\g[t] / t^N \oplus \bC \hbar)_{t^\nu} = (\g[t] / t^N)_{t^\nu}  \oplus \bC \hbar$.

Arguing as in the proof of Theorem \ref{perverse bases for nilcone are the same}, we note that $V_{t^\nu}$ is an irreducible $\torus_\hbar \ltimes G(\bC[t] / t^N)_{t^\nu}$-module, hence the action on it factors through the reductive part, and hence $\ker \ol a_{t^\nu}$ contains $(\g[t]/t^N)^{\mathrm u}_{t^\nu}$ --- the Lie algebra of the unipotent radical of $G(\bC[t] / t^N)_{t^\nu}$.

On the other hand, we have $\Om\vert_{t^\nu} \simeq T_{t^\nu}{O_\nu} \oplus T^*_{t^\nu}{\cW^\la_\nu}$, where $\h_e$ can be identified with $T^*_{t^\nu}{\cW^\la_\nu}$. Let $\fl_0$ be the Lie algebra of homogeneous Poisson derivations of $\cW^\la_\nu$. Then we have the commutative diagram
\[\begin{tikzcd}
	{(\mathfrak{g}[t] / t^N)_{t^{\nu}} } & {\mathfrak h_{t^\nu}} & {T^*_{t^\nu} \mathcal W^\lambda_\nu} & {\mathfrak l_0} \\
	&&& {\mathfrak g_{\nu}}
	\arrow["{\overline \phi_{t^\nu}}", from=1-1, to=1-2]
	\arrow["\simeq", from=1-2, to=1-3]
	\arrow["q", two heads, from=1-3, to=1-4]
	\arrow["j", hook', from=2-4, to=1-1]
	\arrow["i"', hook, from=2-4, to=1-4]
\end{tikzcd}\]
where $q$ is given by Corollary \ref{cotangent space at 0 to g}, $i$ is given by Lemma \ref{lem: centralizer acts faithfully on slice}, and $j$ is the natural inclusion of the reductive part. It follows that $\ker (q \circ \ol \phi_{t^\nu})$ and hence $\ker \phi_{t^\nu}$ is contained in $(\g[t]/t^N)^{\mathrm u}_{t^\nu}$.

We thus obtain $\ker \ol \phi_{t^\nu} \subseteq (\g[t]/(t^N))^{\mathrm u}_{t^\nu} \subseteq \ker \ol a_{t^\nu}$; the theorem follows.
\end{proof}

\begin{rem}
One can easily check that this injective map on classes of simple objects is surjective not on every leaf.

However, Theorem \ref{thm: basis for slice in affine gr} shows that this map becomes an isomorphism on the leaf $\cW_0^{\la'}$ when we consider it as a subvariety of $\cW_0^{\la}$ for a larger $\la$. 
So, we have a filtered diagram given by morphisms preserving perverse bases ${K^{(\mathbb C^\times_\hbar, \Omega_{\mathcal W^{\lambda'}_0})}(\mathcal W^{\lambda'}_0)} \hookrightarrow {K^{(\mathbb C^\times_\hbar, \Omega_{\mathcal W^{\lambda}_0})}(\mathcal W^{\lambda}_0)}$ for $\la \geq \la'$, and its colimit as $\la \to \infty$, is isomorphic to the colimit ${K^{G(\mathcal O) \rtimes \mathbb C^\times}({\mathrm{Gr}})} = \lim_{\la \to \infty} {K^{G(\mathcal O) \rtimes \mathbb C^\times}(\overline{\mathrm{Gr}}^\lambda)}$ with the same perverse basis. 
\end{rem}

\subsection{Other examples} \label{subsec: other examples}
In this Section, we say a few words about other natural examples of symplectic singularities. We do not give an explicit description as detailed as in Sections~\ref{subsec: nilcone},~\ref{subsec: affine grassmannian slices}, but explain how we think one should approach these cases.

\subsubsection{Slodowy slice}
Let $\cN \subset \g$  be the nilpotent cone. Fix a nilpotent $e \in \cN$, let $O_e$ be its $G$-orbit, and $\pi: S_e \hookrightarrow \cN$ the closed embedding of the Slodowy slice to $O_e$ in $\cN$ (by the Slodowy slice we mean not an affine subspace of $\g$, but its intersection with $\cN$).

There is the action Lie groupoid $G \times \cN \rightrightarrows \cN$ on $\cN$. Take its restriction to $S_e$, denote it by $\cG$ (explicitly, it is defined as a subvariety of points $(g, n)$ of $G \times \cN$, such that $n, g(n) \in S_e$). We denote by $\cN_{\geq e}$ the subscheme of $\cN$, which is the union of $G$-orbits on $\cN$, whose closure contains $e$ (these are the orbits, which intersect $S_e$ nontrivially). 
It is easy to see that there is an isomorphism of quotient stacks $[S_e / \cG] \simeq [\cN_{\geq e} / G]$. In particular, the category of $\cG$-equivariant sheaves on $S_e$ is equivalent to the category of $G$-equivariant sheaves on $\cN_{\geq e}$.

The Lie algebroid of the Lie groupoid $\cG$ can be described as follows. There is an action Lie algebroid $\cO_\cN \T \g$ on $\cN$, and one can take its pullback to $S_e$, $\pi^+ (\cO_\cN \T \g)$, which we simply denote by $\pi^+\g$ (see Section~\ref{subsec: direct and inverse image} or \cite[2.4.5]{Kae98} for a discussion of the Lie algebroid pullback). There is a natural functor $\Coh^G(\cN_{\geq e}) \simeq \Coh^{\cG} (S_e) \rightarrow \Coh^{\pi^+\g}(S_e)$, differentiating the action.

\begin{lem} \label{lem: pullback of cotangent algebroid to slodowy}
The pullback of the Poisson Lie algebroid $\Om_{\cN}$ on $\cN$ is isomorphic to the Poisson Lie algebroid on $S_e$: $\pi^+ \Om_\cN \simeq \Om_{S_e}$.
\end{lem}
\begin{proof}
There is the Kazhdan torus action on $\cN$, which contracts $S_e$ to $e$, denote it by $\torus_e$. $\pi^+ \Om_\cN$ is naturally equivariant with respect to $\torus_e$, hence it is determined by its restriction to the formal neighborhood $S_e^{\wedge e}$ at $e$.

Since $S_e$ is a slice, there is a Poisson decomposition
\begin{equation*}
\cN^{\wedge e} \simeq O_e^{\wedge e} \wh \times S_e^{\wedge e},
\end{equation*}
and in particular the decomposition of Lie algebroids:
\begin{equation} \label{eq: decomposition of lie algebroid in formal completion slodowy}
\Om_{\cN}^{\wedge e} \simeq \cT_{O_e}^{\wedge e} \oplus \Om_{S_e}^{\wedge e}.
\end{equation}
From \eqref{eq: decomposition of lie algebroid in formal completion slodowy}, it is clear that $(\pi^+\Om_\cN)^{\wedge e} \simeq (\pi^{\wedge e})^+ \Om_\cN^{\wedge e} \simeq \Om_{S_e}^{\wedge e}$, as required.
\end{proof}

\begin{rem} \label{rem: pullback of poisson sheaves to slice}
The same proof as in Lemma~\ref{lem: pullback of cotangent algebroid to slodowy} works for any (formal) conical slice to a Poisson leaf of a symplectic singularity. It justifies the existence of an inverse image functor from Poisson sheaves on a symplectic singularity to those on a slice. We expect that $!$-restriction should be t-exact w.r.t. perverse t-structures, see also Remark~\ref{rem: !-restriction to slice}.
\end{rem}

There is a surjective morphism $\cO_\cN \T \g \twoheadrightarrow \Om_\cN$ of Lie algebroids on $\cN$, see Section~\ref{subsec: nilcone}. Restricting to $S_e$, we get a surjection of Lie algebroids on $S_e$, $\pi^+ \g \twoheadrightarrow \Om_{S_e}$ (it is straightforward to see that $\pi^+$ preserves surjectivity). It is also straightforward (as in the proofs of Theorems~\ref{perverse bases for nilcone are the same}~and~\ref{thm: basis for slice in affine gr}) to see that this surjection of Lie algebroids induces a functor between categories of perverse coherent sheaves $\cP_\coh^{\Om_{S_e}}(S_e) \rightarrow \cP_\coh^{\pi^+\g}(S_e)$, and that it commutes with IC-extension from any symplectic leaf. 

Note that slices to symplectic leaves of $S_e$ are locally isomorphic to Slodowy slices in the nilcone (see \cite[Section~7.4]{KT21}). Thus, the same argument as in the proof of Theorem~\ref{perverse bases for nilcone are the same} 
shows that given a simple $\cG$-equivariant coherent sheaf on any symplectic leaf of $S_e$, the action of $\pi^+\g$ on it factors through $\Om_{S_e}$.

It is not true in general, however, that simple $\cG$-modules on a leaf are simple as $\pi^+\g$-modules. For example, on the closed leaf $\{e\}$, simple $\cG$-modules are simple representations of the stabilizer $G_e$ of $e$ in $G$, while simple $\pi^+\g$-modules are simple representations of its Lie algebra $\g_e$. If $G_e$ is not connected, the simples differ.

So we have K-groups $K^{\cG}(S_e) = K^{G}(\cN_{\geq e})$, $K^{\pi^+\g}(S_e)$ and $K^{\Om_{S_e}}(S_e)$ with perverse bases. They are not isomorphic in general, but are very closely related. We believe one should be able to add some corrections to these K-groups (and categories), so that they become isomorphic (with bases coinciding); possibly this should involve the homotopy groupoid of the orbit $O_e$, see Remark~\ref{rem: !-restriction to slice}.

So we explained that the perverse basis in $K^{\Om_{S_e}}(S_e)$ should be related to the perverse basis in $K^{G}(\cN_{\geq e})$, which in turn is part of the affine Kazhdan--Lusztig basis for $G^\vee$. One can count this as evidence that for a general symplectic singularity, our basis should enjoy some KL-type properties.


\subsubsection{Affine Grassmannian slice to a nonzero orbit}
In Section~\ref{subsec: affine grassmannian slices}, we related the perverse basis for $\cW_0^\la$ to the perverse basis for $\ol \Gr^\la$. Here we briefly discuss a possible approach to the slice $\cW^\la_\mu$ for $\mu \neq 0$.

Let $\pi: \cW^\la_\mu \hookrightarrow \ol \Gr^\la$ be the locally closed embedding. Recall the Lie algebra action of $\g[t] \oplus \bC \hbar$ on $\ol \Gr^\la$ (see Section~\ref{subsec: affine grassmannian slices}). 
Similarly to the case of Slodowy slice in Lemma~\ref{lem: pullback of cotangent algebroid to slodowy} and its surrounding discussion, we obtain a surjection $\pi^+ (\g[t]\oplus \bC \hbar) \twoheadrightarrow \Om_{\cW^\la_\mu} \oplus \cO_{\cW^\la_\mu}\hbar$ of Lie algebroids on $\cW^\la_\mu$. We believe that similarly to Theorem~\ref{thm: basis for slice in affine gr}, the simple modules in categories of perverse coherent sheaves, equivariant w.r.t. these Lie algebroids, should be closely related.

On the other hand, similarly to Proposition~\ref{prop: from schubert to slice}, there should be a close relation between categories $\Coh^{\pi^+ (\g[t] \oplus \bC \hbar)}(\cW^\la_\mu)$ and $\Coh^{G(\cO) \rtimes \torus_\hbar}(\Gr^{\mu \leq \cdot \leq \la})$, (here $\Gr^{\mu \leq \cdot \leq \la}$ denotes the union of $G(\cO)$-orbits that intersect $\cW^\la_\mu$ nontrivially).

So, we believe, there should be a close relation between the perverse bases of $K^{\Om_{\cW^\la_\mu} \oplus \cO_{\cW^\la_\mu} \hbar}(\cW^\la_\mu)$ and $K^{G(\cO) \rtimes \torus_\hbar}(\Gr^{\mu \leq \cdot \leq \la})$, which is a part of the perverse coherent basis for the affine Grassmannian.

\subsection{Towards perverse coherent sheaves on double affine Grassmannian} \label{subsec: double affine grass}

This subsection contains no mathematical statements.

As we showed in Proposition~\ref{prop: from schubert to slice}, one can recover the category of (perverse) coherent sheaves on $\ol \Gr^\la$
from its restriction to the maximal transversal slice $\cW^\la_0$, by remembering the additional structure of being $G(\cO) \rtimes \bC^\times$-equivariant on $\ol \Gr^\la$. This approach may be useful for defining the notion of (perverse) coherent sheaves on the double affine Grassmannian, or more generally, the affine Grassmannian of a Kac--Moody group. 
The issue is, while this space is defined only as a prestack (see \cite{BV25}), the transversal slices in it are believed to be the Coulomb branches of quiver gauge theories (see \cite{Fin18}, \cite[3(viii)]{BFN19}). In this subsection, we speculate about which additional structure we should impose on sheaves on Coulomb branches, so that these sheaves should be interpreted as current-group-equivariant sheaves on the double affine Grassmannian.

First, we recall and explain what happens in the finite type case (where the affine Grassmanian is a defined ind-scheme). 
Recall the Lie algebra $\g(\cO)$ action on $\cW_0^\la$ (Section~\ref{subsec: affine grassmannian slices}). We now explain how to construct this $\g(\cO)$-action without appealing to an embedding into the projective variety $\ol \Gr^\la$.

We explain how to define the action of $\g(\cO)$ on $\cW_0$, from which the action on $\cW^\la_0$ comes by restriction. Equivalently, we need to define an action of $U(\g[t])$ on $\bC[\cW_0]$. Note that $U(\g[t])$ and $\bC[\cW_0]$ are two degenerations (commutative and cocommutative) of the Yangian $Y(\g)$, and the desired action is nothing but the (infinitesimal) dressing action for the Manin triple $(\g[t], t^{-1}\g[[t^{-1}]], \g((t^{-1})))$. One way to define it is as follows. The Hopf algebra $Y(\g)$ admits the action on itself by conjugation; its dequantization yields the dressing action $U(\g[t]) \curvearrowright \bC[\cW_0]$,  see \cite[Theorem 3.10, Definition 3.11]{Lu93}.

The key feature used in this construction is the coproduct structure on $Y(\g)$, needed to define the adjoint action (we expect that if one takes $\cW^\la_\mu$ instead of $\cW^\la_0$, one would need the coproduct for shifted Yangians of \cite{FKPRW18}). The substitute for coproduct for quantum Coulomb branches beyond the finite type case has not yet been constructed, but is expected to exist by experts in the area. One can expect that once defined, it may be used to generalize the construction above to arbitrary Kac--Moody type.

In parallel with Theorem~\ref{thm: basis for slice in affine gr}, we expect that perverse coherent sheaves on Coulomb branches of quiver gauge theories, as defined in Section~\ref{subsec: definition of the category}, should be related to sheaves with equivariance conjecturally described above.

Finally, let us mention that if such an approach is possible, it could be applied not only to coherent sheaves, but also to any reasonable category of sheaves — in particular, to equivariant constructible perverse sheaves — and might be used to better understand the geometric Satake correspondence for Kac--Moody groups \cite[3(viii)]{BFN19}, \cite{BV25}.

\end{document}